\documentclass[12pt]{elsarticle}
\usepackage[top=1.3in, bottom=2cm, left=1.5cm, right=2.5cm]{geometry}
\usepackage{graphicx,epsfig}
\usepackage{caption}
\usepackage{subcaption}
\usepackage{amssymb,amsmath}
\usepackage{graphicx}
\usepackage{epstopdf}

\journal{}%

\begin{document}
\begin{frontmatter}
\newtheorem{theorem}{Theorem}[section]
\newtheorem{lemma}{Lemma}[section]
\newtheorem{corollary}{Corollary}[section]
\newdefinition{definition}{Definition}
\newdefinition{remark}{Remark}%[thm]
\newproof{proof}{Proof}
\newdefinition{example}{Example}%[thm]
\title{Region unknotting number of $2$-bridge knots}
%\tnotetext[t1]{This document is a collaborative effort.}
%\tnotetext[\dag]{Corresponding Author}
\author{Vikash. S}%\fnref{fn1}
%\ead{vikash@iitrpr.ac.in}
\author{Madeti. P}%\fnref{fn2}
%\ead{prabhakar@iitrpr.ac.in}
%\ead[url]{http://www.iitrpr.ac.in/html/faculty/prabhakar.shtml}
%\cortext[cor1]{Corresponding author}
%\cortext[cor2]{Principal corresponding author}
%\fntext[fn1]{This is the specimen author footnote.}
%\fntext[fn2]{Another author footnote, but a little more longer.}
%\fntext[fn3]{Yet another author footnote. Indeed, you can have
%any number of author footnotes.}
%\address[v]{Department of Mathematics, IIT Ropar, Rupnagar- 140001, India.}
\address{Department of Mathematics, IIT Ropar, Rupnagar - 140001, India.}

%%%%%%%%%%%%%%%%%%%%%%%%%%%%%%%%%%%%%%%%%%%%%%%%%%%%%%%%%%%%%%%%%%%%%%%
%%%%%%%%%%%%%%%%%%%%%%%%%%%%%%%%ABSTRACT%%%%%%%%%%%%%%%%%%%%%%%%%%%%%%%
%%%%%%%%%%%%%%%%%%%%%%%%%%%%%%%%%%%%%%%%%%%%%%%%%%%%%%%%%%%%%%%%%%%%%%%
\begin{abstract}
In this paper, we discuss the region unknotting number of different classes of 2-bridge knots. In particular, we provide region unknotting number for the classes of $2$-bridge knots whose Conway notation is $C(m,\ n), C(m,\ 2,\ m),$ $ C(m,\ 2,\ m\pm1)$ and $C(2,\ m,\ 2,\ n)$. By generalizing, we also provide a sharp upper bound for all the remaining classes of $2$-bridge knots.
\end{abstract}

%Keywords: $2$-bridge knots; Region unknotting number; Proper links; Arf invariant\\
%
%Mathematics Subject Classification 2010: 57M25, 57M27

\end{frontmatter}

%%%%%%%%%%%%%%%%%%%%%%%%%%%%%%%%%%%%%%%%%%%%%%%%%%%%%%%%%%%%%%%%%%%%%%%
%%%%%%%%%%%%%%%%%%%%%%%%%%%%%%%%INTRODUCTION%%%%%%%%%%%%%%%%%%%%%%%%%%%
%%%%%%%%%%%%%%%%%%%%%%%%%%%%%%%%%%%%%%%%%%%%%%%%%%%%%%%%%%%%%%%%%%%%%%%
\section{Introduction}
\noindent In \cite{ayaka}, A. Shimizu introduced a new local transformation on link diagrams and named it as \textit{region crossing change}. In \cite{ayaka, proper}, it was proved that this new local transformation is an unknotting operation for a knot or a proper link. Here a region crossing change at a region $R$ of a knot diagram $D$ is defined to be the crossing changes at all the crossing points on $\partial R$. The region unknotting number $u_R(D)$ of a knot diagram $D$ is the minimum number of region crossing changes required to transform $D$ into a diagram of the trivial knot without Reidemeister moves. The region unknotting number $u_R(K)$ of $K$ is defined to be the minimal $u_R(D)$ taken over all minimal crossing diagrams $D$ of $K$. In \cite{proper}, Z. Cheng proved that region crossing change for a link is an unknotting operation if and only if the link is proper.\\

\noindent Many knot theorists studied different unknotting operations like $\sharp$-operation  \cite{hash},
 $\delta$-operation~\cite{delta}, $3$-gon operation~\cite{3-gon}, H(n)-operation~\cite{H(n)}
 and $n$-gon~\cite{n-gon} operations. In \cite{3-gon} Y. Nakanishi proved that a $\delta$-unknotting operation can be obtained from a finite sequence of $3$-gon moves. In \cite{n-gon}, H. Aida generalized $3$-gon moves to $n$-gon moves and proved that an $n$-gon move is an unknotting operation.\\

\noindent  It is interesting to observe that both $\sharp$-operation and $n$-gon moves are special cases of region crossing change. Finding region unknotting number for different knots is a challenging problem. In \cite{ayaka}, A. Shimizu showed that for a twist knot $K$, $u_R(K)=1$ and for torus knots of type $K(2,4m\pm1)$, $u_R(K(2,4m\pm1))=m$, where $m\in \mathbb{Z^+}$. In \cite{regionunko}, we provided a sharp upper bound for region unknotting number of torus knots.\\

\noindent In this paper, we provide region unknotting number for all those $2$-bridge knots whose Conway's notation is $C(m,\ n),\ C(m,\ 2,\ n)$, $C(m,\ 2,\ m\pm1)$ and $C(2,\ m,\ 2,\ n)$. We also discuss some bounds on region unknotting number for other $2$-bridge knot classes. Since minimal crossing diagrams are required to find region unknotting number of knots, we mainly look for all the $2$-bridge knot diagrams with minimum crossings. In this context,  it is required to observe that all $2$-bridge knots are prime \cite{Schubert} and alternating \cite{goodrick1972}. Specifically, using Tait's third conjecture, which is true \cite{tait3}, one can obtain all the minimal diagrams of a prime reduced alternating knot $K$ from a minimal crossing diagram of $K$ by performing finite number of flypings. In \cite{ayaka}, A. Shimizu provided a method to find all possible minimal crossing diagrams of a prime alternating knot. Here our concentration is only on the $2$-bridge knots.\\

\noindent Based on the method provided in \cite{ayaka}, we can show a $2$-bridge knot $C(m,\ n)$, where ($m,\ n\neq 0 \in \mathbb{Z}, mn>0$), has only one minimal crossing diagram on $S^2$. Let $D$, as in Figure \ref{mn}(a), be a minimal crossing diagram of a $2$-bridge knot $C(m,\ n)$. For every crossing point $c$ in integer tangle $t_m$ (horizontal tangle having $m$ half twists), $T_c^+$ and $T_c^-$ are shown in Figure \ref{mn}(b) and \ref{mn}(c), respectively. To get non-trivial flyping and hence non equivalent minimal diagrams of $C(m,\ n)$, $T_c^+$ and $T_c^-$ should not satisfy any of the following three conditions:
\begin{enumerate}
\item the tangle $T_c^\epsilon$ is not a tangle sum ($\epsilon = +,-$)
\item the tangle $T_1$ or $T_2$ is an integer $2$-tangle% i.e., the sum of some tangles $\epsilon1$ $\epsilon = +,-$
\item the tangles $T_1$ and $T_2$ satisfy $T_{1hv} = T_1$ and $T_{2v} = T_2$, or $T_{1v} = T_1$ and $T_{2hv} = T_2$.
\end{enumerate}
Observe that the tangle $T_c^+$ is sum of two tangles $T_1$ and $T_2$, where either $T_1$ is $t_i$ and $T_2$ is tangle sum of $t_n'$ (vertical tangle having $n$ half twists) and $t_{m-(i+1)}$ or $T_1$ is tangle sum of $t_i$ and $t_n'$ and $T_2$ is $t_{m-(i+1)}$ (where $0\leq i \leq m-1$). The tangle $T_c^-$ is not a tangle sum. Then $T_c^+$ and $T_c^-$ satisfy  the cases (2) and (1) respectively. Therefore we can not perform non-trivial flyping on any $c$ and $T_c^\epsilon$, where $c$ is a crossing in $m$-tangle. Since $C(m,\ n) \sim C(n,\ m)$, same happens for any crossing $c$ from vertical tangle $t_n'$. Hence, $2$-bridge knot $C(m,\ n)$ has only one minimal diagram.\\

\begin{figure}
\begin{center}
\begin{subfigure}[b]{0.2\textwidth}
                \includegraphics[width=\textwidth]{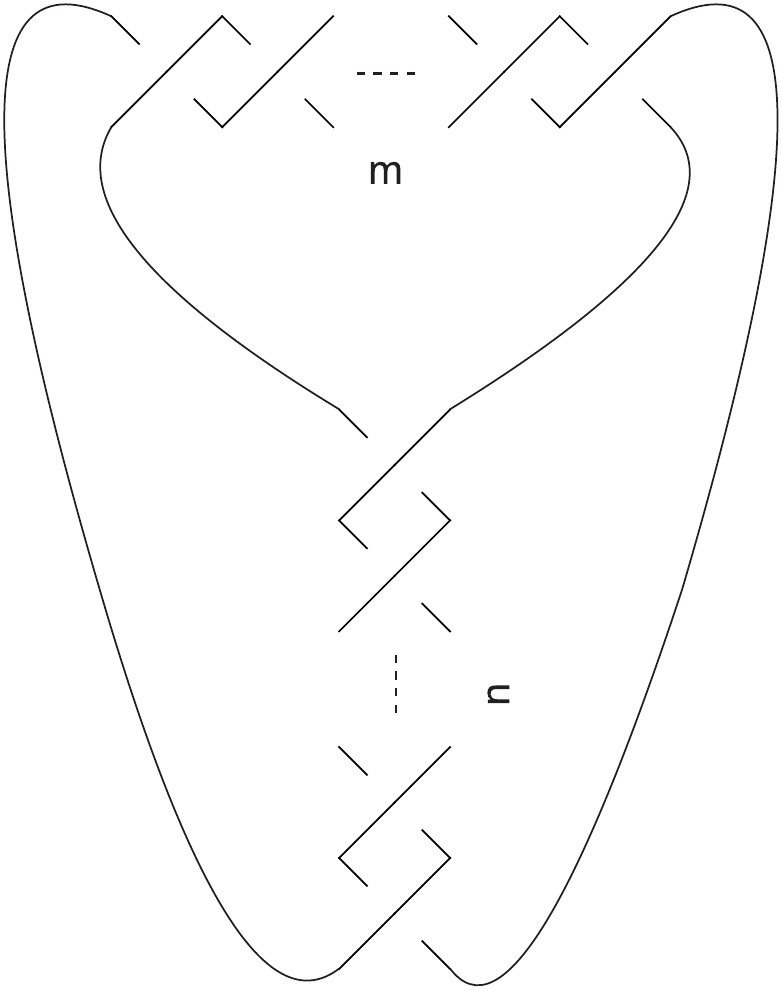}
                \caption{}
                \label{cmn}
        \end{subfigure}%
        ~ %add desired spacing between images, e. g. ~, \quad, \qquad, \hfill etc.
          %(or a blank line to force the subfigure onto a new line)
        \begin{subfigure}[b]{0.35\textwidth}
                \includegraphics[width=\textwidth]{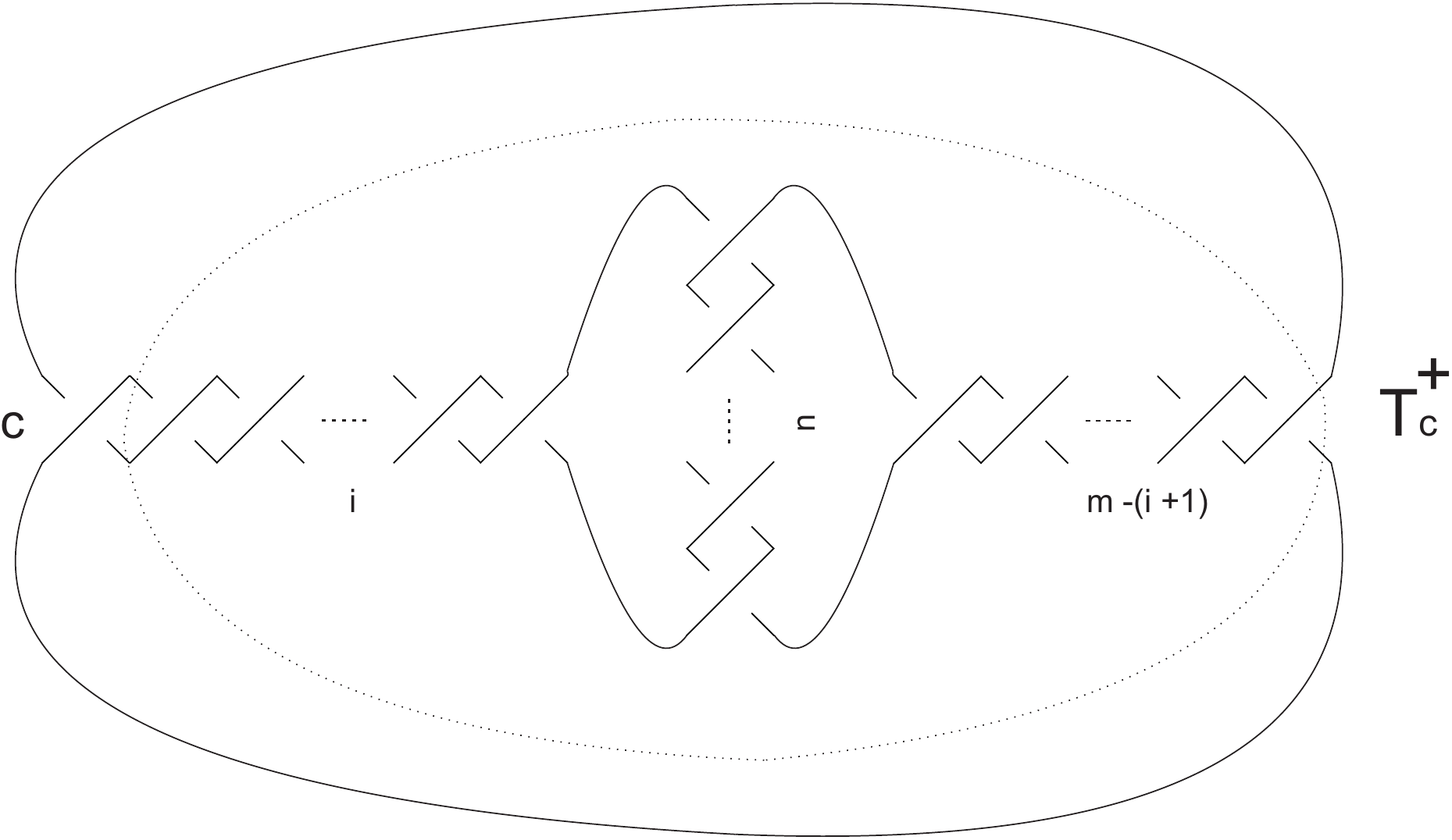}
                \caption{}
                \label{mna}
        \end{subfigure}
        ~ %add desired spacing between images, e. g. ~, \quad, \qquad, \hfill etc.
          %(or a blank line to force the subfigure onto a new line)
        \begin{subfigure}[b]{0.25\textwidth}
                \includegraphics[width=\textwidth]{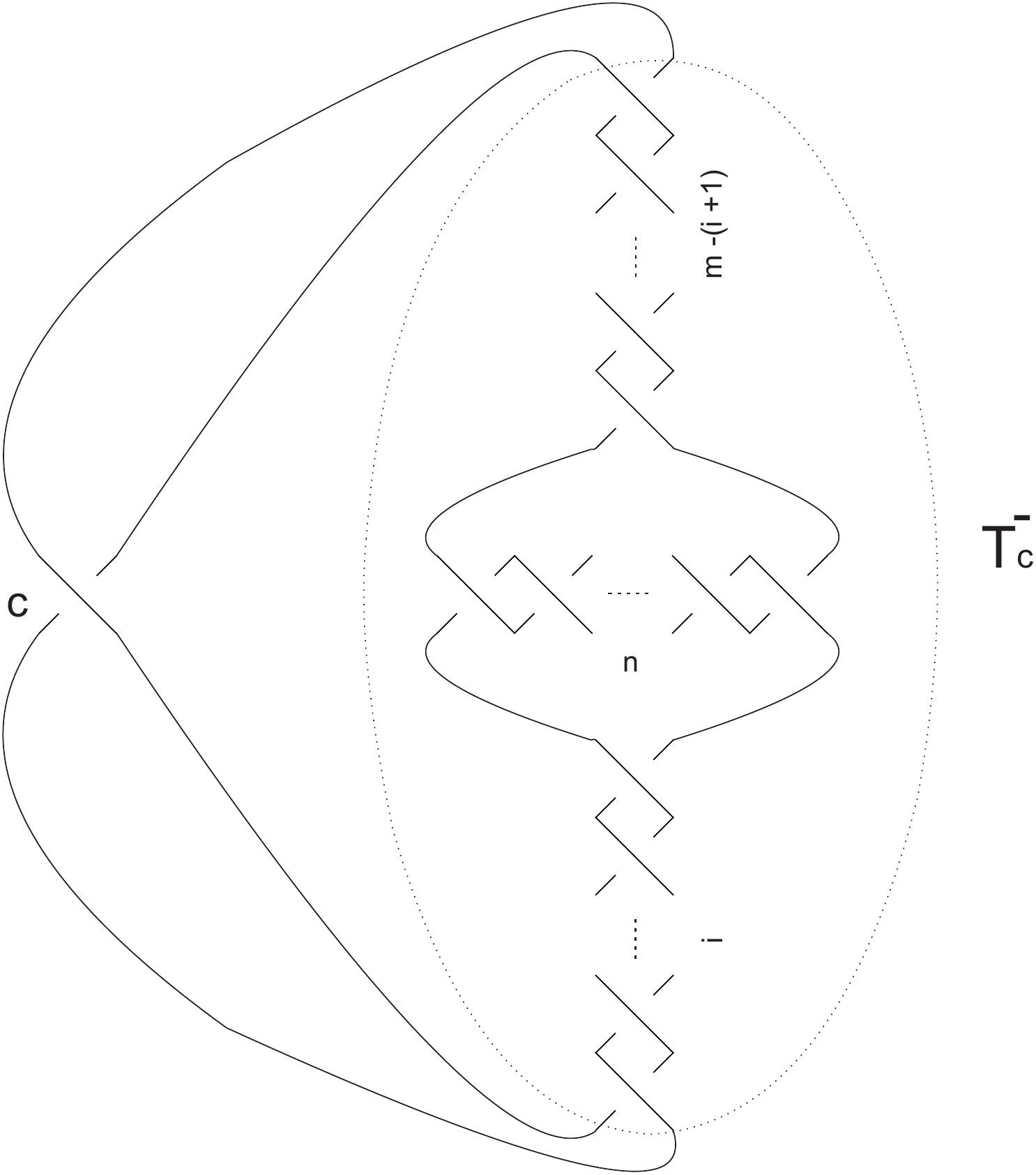}
                \caption{}
                \label{mnb}
        \end{subfigure}
%        \caption{Pictures of animals}\label{fig:animals}
%\includegraphics[width=5cm,height=1cm]{parallel.pdf}
\vspace*{8pt}
\caption{Minimal diagram for 2-bridge knot $C(m,\ n)$}\label{mn}
\end{center}
\end{figure}

\noindent In Section \ref{ubfrunf2bk}, we provide region unknotting number for $2$-bridge knot classes whose Conway's notation is $C(m,\ n),$ $C(m,\ 2,\ m)$ and $C(m,\ 2,\ m\pm1)$. Also we provide an upper bound for region unknotting number for all $2$-bridge knots. In Section \ref{arfinvariant}, we provide Arf invariant for $2$-bridge knots (not links) whose Conway's notation is $C(m,\ n)$ and $C(m,\ p,\ n)$. %Also we observe that Arf invariant for links depends on the orientation of its components.

%%%%%%%%%%%%%%%%%%%%%%%%%%%%%%%%%%%%%%%%%%%%%%%%%%%%%%%%%%%%%%%%%%%%%%%
%%%%%%%%%%%%%%%%%%%%%%%%%%%%%%%%SECTION-2%%%%%%%%%%%%%%%%%%%%%%%%%%%%%%
%%%%%%%%%%%%%%%%%%%%%%%%%%%%%%%%%%%%%%%%%%%%%%%%%%%%%%%%%%%%%%%%%%%%%%%
\section{Region unknotting number for $2$-bridge knots}\label{ubfrunf2bk}

\noindent In this section we provide region unknotting number of $2$-bridge knot $K$ whose Conway notation is $C(m,\ n)$ by showing $u_R(K) = u_R(C(m,\ n))$. Region unknotting number of $2$-bridge knots whose Conway notation is  $C(m,\ 2,\ m)$ and $C(m,\ 2,\ m\pm1)$ is shown to be one. Also we give upper bound for $2$-bridge knot classes whose Conway notation is $ C(m,\ 2,\ n)$, $C(m,\ p,\ n)$, $C(c_1,\ c_2, \cdots, c_n)$ where $c_{2k+1} \text{ is } even$ and $C(c_1,\ c_2, \cdots, c_n)$ where $c_{2k} \text{ is } even$ and $n \text{ is } even$. At last, a general upper bound for region unknotting number for  all $2$-bridge knots is also provided.\\

\noindent The key idea to ensure the region unknotting number is that in a $2$-bridge knot $C(c_1\ c_2 \cdots c_n)$ each tangle $c_i$ is a $(2,q)$ type toric braid and by \cite{regionunko}, region unknotting number of $(2,q)$ type torus knot or proper link is $ \lfloor\frac{q+2}{4}\rfloor$. Hence, to convert an integer $2$-tangle $t_n$ or $t_n'$ into $0$ or $\infty$ tangle, respectively, we need to make atleast $ \lfloor\frac{n+2}{4}\rfloor$ region crossing changes. Throughout this paper, we consider only those $2$-bridge knots which are either knots or proper links. Observe that in $C(m,\ n)$, if both $m$ and $n$ are odd and $m+n\not\equiv 0\ (mod\  4)$, then $2$-bridge knot $C(m,\ n)$ is not proper.

\begin{theorem}\label{5.thm23}
Let K be a $2$-bridge knot/proper link whose Conway's notation is $C(m,n)$. Then we have the following:
\begin{enumerate}
\item if $m,n$ are even, then $u_R(K)=\lfloor\frac{{min\{m,n\}+2}}{4}\rfloor$,
\item if $m$ even, $n$  odd, then $u_R(K)=\lfloor\frac{m+2}{4}\rfloor$,
\item if $m$ odd, n even, then  $u_R(K)=\lfloor\frac{n+2}{4}\rfloor$,
\item if $m,n$ are odd, then $u_R(K)=\frac{m+n}{4}$.%, if $m+n\equiv 0\ (mod\  4)$.
\end{enumerate}
\end{theorem}

\begin{proof}
Since $K$ is a $2$-bridge knot with Conway notation $C(m,\ n)$, the only minimal diagram for $K$ is as shown in Figure \ref{regiondata2-bridge}(a). From Figure \ref{regiondata2-bridge}(a), it is clear that this minimal diagram of $K$ has total $m+n+2$ regions, out of which the regions $R_1, R_{m+1}; R_1',R_{n+1}';$ and the remaining $m+n-2$ regions have $n+1; m+1;$ and $2$ crossings respectively on their boundaries.\\

\noindent To get trivial knot diagram from $C(m,\ n)$, we need to make region crossing changes such that sum of signs of crossings in either horizontal or vertical tangle become $0$. In other words, to transform $C(m,\ n)$ to unknot by region crossing changes, we need to reduce either $m$ or $n$ to $0$. In this process of selection of regions, observe that a region crossing change at any one of $R_1, R_{m+1} \text{ or } R_1',R_{n+1}'$ in $C(m,\ n)$ will reduce $m$ to $m-2$ or $n$ to $n-2$ respectively and hence at each step, the absolute value of the sum of signs of crossings of either horizontal or vertical tangle reduce by $2$. But the region crossing change at any other region will reduce sum of signs of crossings of either horizontal or vertical tangle by $4$.\\

\noindent Since  the choice of regions is based on the values of $m$ and $n$, here we provide region unknotting number of $C(m,\ n)$ for all possible cases of $m$ and $n$.\\

\noindent \textbf{\underline{Case (i)}} If both $m$ and $n$ are even:\\
Without loss of generality assume $n\leq m$. If $n\equiv 0\ (mod\  4)$, then make region crossing changes at any non-consecutive $\frac{n}{4}$ regions among $R'_j(2\leq j\leq n )$. These region crossing changes reduce the absolute value of sum of signs of crossings of vertical tangle to zero i.e., the diagram $C(m,\ n)$ transforms to a diagram of $C(m,\ 0)$, which is $m$ times twisted unknot. Hence $u_R(K)\leq \frac{n}{4}$. Since it is not possible to reduce a $t_n'$ tangle to $\infty$ tangle with less than $ \frac{n}{4} $ region crossing changes,  $u_R(K) = \frac{n}{4}=\lfloor\frac{n+2}{4}\rfloor$.\\

\begin{figure}[h]
\begin{center}
\begin{subfigure}[b]{0.2\textwidth}
                \includegraphics[width=\textwidth]{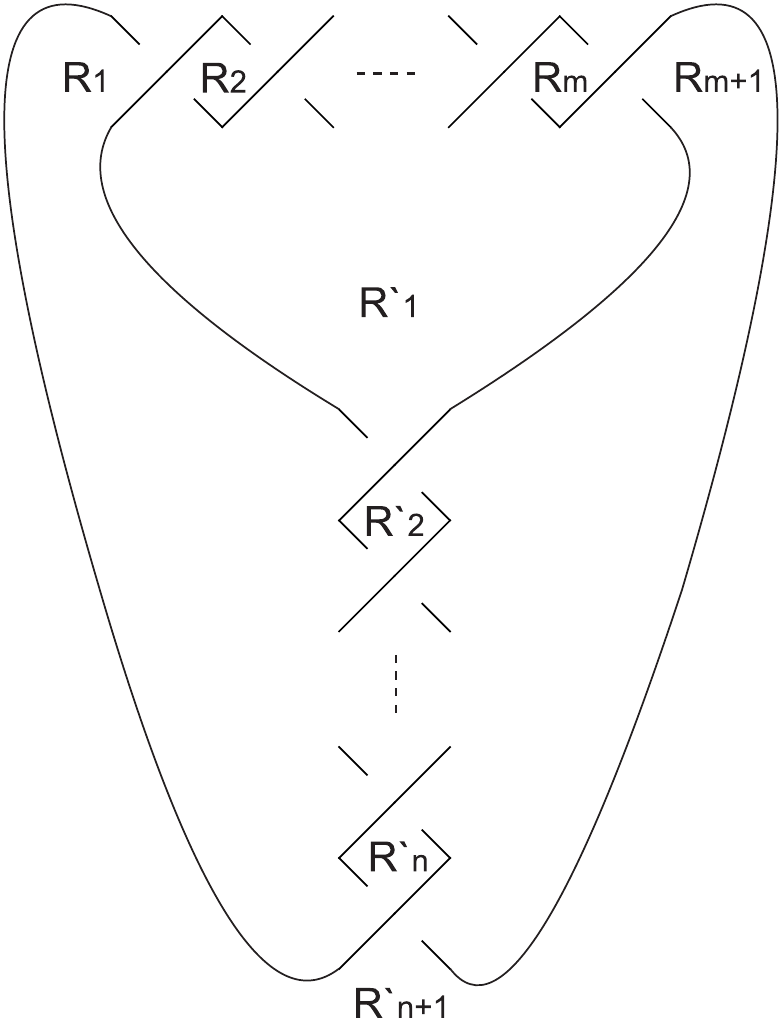}
                \caption{}
                \label{mndata}
        \end{subfigure}~~~~~
       \begin{subfigure}[b]{0.4\textwidth}
                \includegraphics[width=\textwidth]{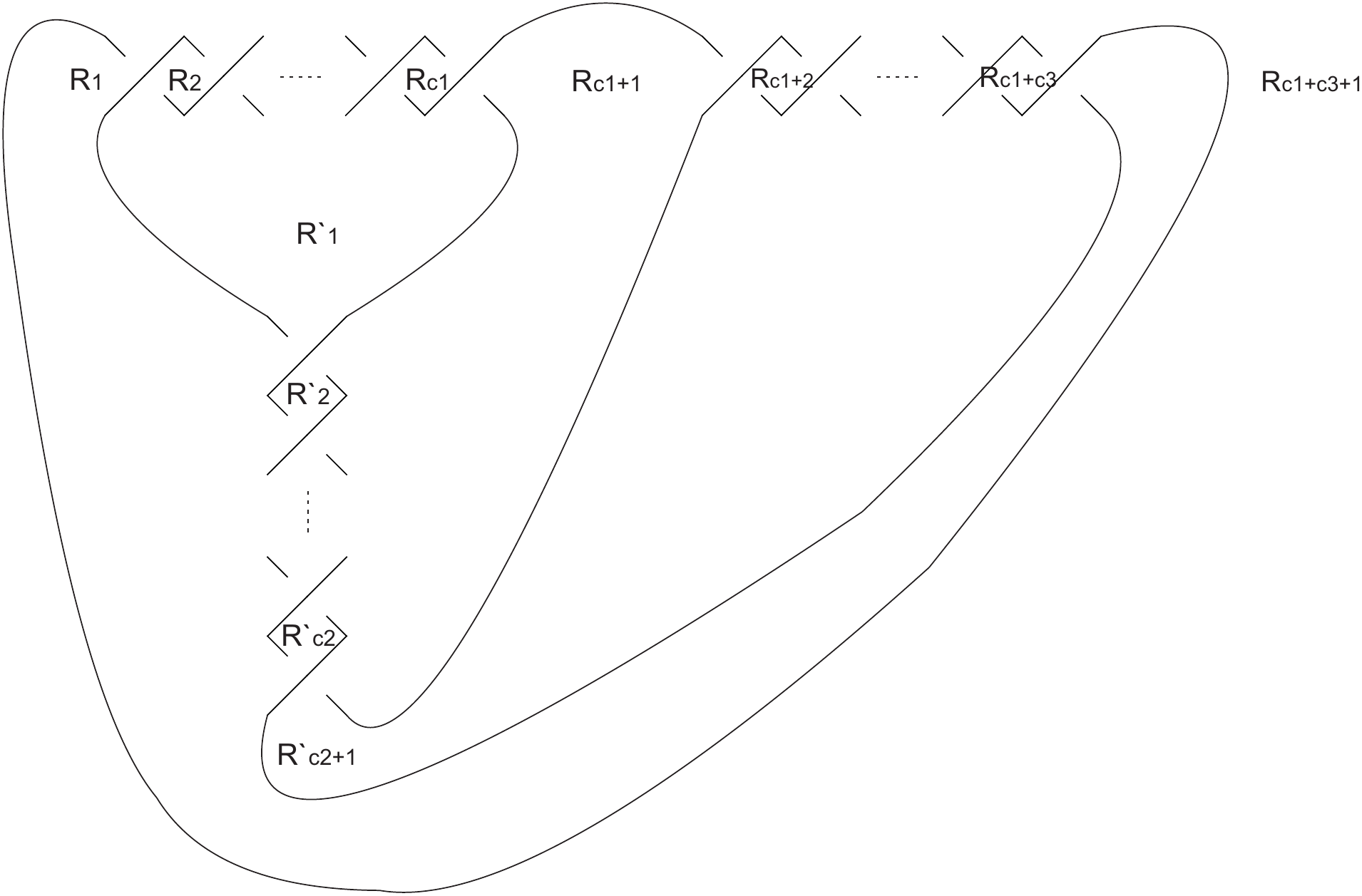}
                \caption{}
                \label{odddata}
        \end{subfigure}
\vspace*{8pt}
\caption{Region Data for 2-bridge knot}\label{regiondata2-bridge}
\end{center}
\end{figure}

\noindent If $n \not\equiv 0\ (mod\  4)$, then region crossing changes at any non-consecutive  $\frac{n-2}{4}$ regions of  $R'_j(3\leq j\leq n )$, transforms the diagram $C(m,\ n)$ into $C(m,\ 2)$. Note that each of these region crossing change reduces the sum of signs of crossings by $4$. Then region crossing change at any non-consecutive  $\frac{n-2}{4}$ regions of  $R'_j(3\leq j\leq n )$ and  $R'_1$ in $C(m,\ n)$ results in a trivial knot diagram. It is easy to observe that these are the minimum number of regions required to convert $C(m,\ n)$ to a trivial knot diagram. Hence,  $u_R(K)= \frac{n-2}{4} +1 = \lfloor\frac{n+2}{4}\rfloor$\\

\noindent \textbf{\underline{Case (ii)}}  If $m$ is even, $n$ is odd:\\
Observe that by making region crossing changes at any non-consecutive $\lfloor\frac{n+2}{4}\rfloor$ regions from $R'_j (2\leq j\leq n)$ in $t_n'$ tangle, the resultant diagram will be a diagram of either $(2,m+1)$ or $(2,m-1)$ torus knot. Since $u_R(2,q) = \lfloor\frac{q+2}{4}\rfloor$, $u_R(K)\leq \lfloor\frac{n+2}{4}\rfloor + \lfloor\frac{m+3}{4}\rfloor$ or $ u_R(K)\leq \lfloor\frac{n+2}{4}\rfloor + \lfloor\frac{m+1}{4}\rfloor$.
Note that, here we make region crossing changes in both $t_n'$ and $t_m$ tangles. But if we first make region crossing changes in $t_m$ as in case (i), we get $u_R(K) \leq \lfloor\frac{m+2}{4}\rfloor $.\\

\noindent Since these are the only possibilities of choices of regions to convert $K$ to an unknot and since $\lfloor\frac{m+2}{4}\rfloor \leq \lfloor\frac{n+2}{4}\rfloor + \lfloor\frac{m+3}{4}\rfloor$ and $\lfloor\frac{n+2}{4}\rfloor + \lfloor\frac{m+1}{4}\rfloor$, we have $u_R(K) = \lfloor\frac{m+2}{4}\rfloor $.\\

\noindent \textbf{\underline{Case (iii)}}  If $n$ is even and $m$ is odd:\\  Proof is similar to case when $m$ is even and $n$ is odd as $C(m,\ n) = C(n,\ m)$. In this case $u_R(K) = \lfloor\frac{n+2}{4}\rfloor $.\\

\noindent \textbf{\underline{Case (iv)}}  If both $m$ and $n$ are odd:\\ It is easy to observe that neither $m$ nor $n$ separately can reduce to $0$. Using the same procedure as in case 2, if  $n \equiv 1\ (mod\ 4)$, we get $ u_R(K) = \lfloor\frac{n+2}{4}\rfloor + \lfloor\frac{m+3}{4}\rfloor = \frac{m+n}{4} $.\\

\noindent If $n \equiv -1\ (mod\ 4)$, then $ u_R(K) = \lfloor\frac{n+2}{4}\rfloor + \lfloor\frac{m+1}{4}\rfloor = \frac{m+n}{4} $. Hence $u_R(K)=\frac{m+n}{4}$. \hfill $\square$
\end{proof}

\noindent To provide region unknotting number for $2$-bridge knots of type $C(m,\ 2,\ m),\ C(m,\ 2,\ m\pm1)$, first we provide an upper bound for region unknotting number for a general class of $2$-bridge knot whose Conway's notation is $C(m,\ 2,\ n)$. Note that the 2-bridge knot $C(m,\ 2,\ n)$ is a $2$-component link $L=K_1 \cup K_2$ iff $m\equiv n\ (mod\ 2)$. Also $lk(K_1,\ K_2)=$
$\begin{cases}
\frac{m+n}{2}& \text{\emph{if both $m$ and $n$ are even};}\\
\frac{m+n+2}{2}& \text{\emph{if both $m$ and $n$ are odd}}.
\end{cases}$.\\
It is easy to calculate that, in both the cases, link $L$ will be proper iff $m\equiv n\ (mod\ 4)$. For $2$-bridge knots and proper links $C(m,\ 2,\ n)$, we  have the following upper bound.

\begin{theorem}\label{5.thm27}
For 2-bridge knot $K$ with Conway's notation $C(m,\ 2,\ n)$,
\[u_R(K)\leq \left\lfloor\frac{|m-n|+2}{4}\right\rfloor+1.\]
\end{theorem}

\begin{proof}
Observe that after a region crossing change at region $R$ as in  Figure \ref{m2n}, the resultant diagram will be a diagram of $(2,\ n-m)$-type torus knot/link. Since region unknotting number for $(2,\ n-m)$ torus knot/link is $\left\lfloor\frac{|m-n|+2}{4}\right\rfloor$, it is easy to observe that region crossing changes at any non-consecutive $\left\lfloor\frac{|m-n|+2}{4}\right\rfloor$  regions in $t_m$ (if $m>n$) or in $t_n$ (if $n>m$) together with region crossing change at $R$ in $C(m,\ 2,\ n)$ provide trivial knot diagram. Thus \[u_R(K)\leq \left\lfloor\frac{|m-n|+2}{4}\right\rfloor+1.\] \hfill $\square$
\begin{figure}
\begin{center}
\includegraphics[width=7cm,height=4cm]{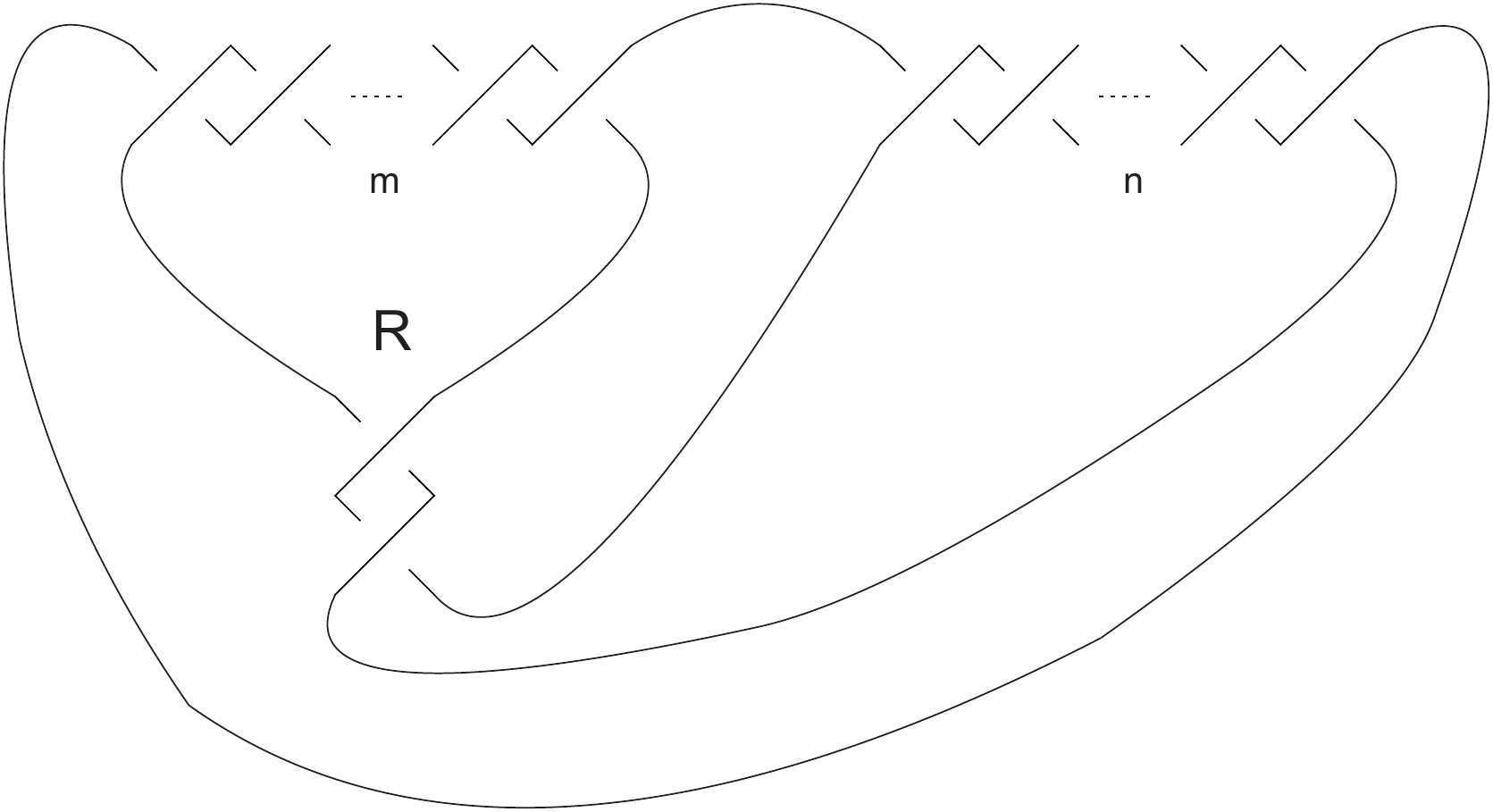}
\vspace*{8pt}
\caption{$C(m,\ 2,\ n)$}\label{m2n}
\end{center}
\end{figure}
\end{proof}

\begin{corollary}
Region unknotting number for 2-bridge knot/link $C(m,\ 2,\ m)$ and $C(m,\ 2,\ m\pm1)$ is one.
\end{corollary}
\begin{proof}
It is clear from Theorem \ref{5.thm27} that if $n = m$ or $m\pm1$ then $\left\lfloor\frac{|m-n|+2}{4}\right\rfloor = 0$. Hence $u_R(C(m,\ 2,\ n)) = 1$.
\end{proof}

\noindent In case of $2$-bridge link $C(m,\ p,\ n)$, we consider different cases depending on the values of $m$, $n$ and $p$. Note that $C(m,\ p,\ n)$ is not proper in the following cases:
\begin{enumerate}
\item when both $m$ and $n$ are even and $m\not\equiv n\ (mod\ 4)$
\item when both $m$ and $n$ are odd and $p$ is even and $m+n+p\equiv 2\ (mod\ 4)$.
\end{enumerate}

\begin{theorem}\label{5.thm28}
For 2-bridge knot $K$ with Conway's notation $C(m,\ p,\ n)$, where either $m$ or $n$ is even,
\[u_R(K)\leq \left\lfloor\frac{m+n+2}{4}\right\rfloor.\]
\end{theorem}

\begin{proof}
Without loss of generality, assume that $m$ is even. After region crossing changing at the regions $\langle R_3, R_7, \cdots,$ $R_{3 + 4\cdot \lfloor\frac{m-2}{4}\rfloor}\rangle$, the resultant diagram is a diagram of either $(2, n)$ or $(2,n-2)$ type torus knot based on whether $m\equiv\ 0\ (mod\ 4)$ or  $m\equiv\ 2\ (mod\ 4)$. Since region unknotting number for $(2, q)$ type torus knot is $\lfloor\frac{q+2}{4}\rfloor$, the region crossing changes at $\langle R_3, R_7, \cdots, R_{3 + 4\cdot \lfloor\frac{m-2}{4}\rfloor}, \cdots, R_{3 + 4\cdot \lfloor\frac{m+n-2}{4}\rfloor}\rangle$ regions transform $C(m,\ p,\ n)$ to a diagram of trivial knot. Hence the number of region crossing changes to unknot $C(m,\ p,\ n)$ is either $\lfloor\frac{m+2}{4}\rfloor + \lfloor\frac{n+2}{4}\rfloor$ or $\lfloor\frac{m+2}{4}\rfloor + \lfloor\frac{n}{4}\rfloor$ based on whether $m\equiv\ 0\ (mod\ 4)$ or  $m\equiv\ 2\ (mod\ 4)$. Hence \[u_R(K)\leq \left\lfloor\frac{m+n+2}{4}\right\rfloor.\] \hfill $\square$
\end{proof}

\begin{remark}
From Theorem \ref{5.thm28}, it is easy to observe that $u_R(C(2,\ p,\ 3)) = 1$ for any $p$. As shown in Figure \ref{2p3}, region crossing change at  $R_3$, results in a trivial knot diagram.
\begin{figure}
\begin{center}
\includegraphics[width=4cm,height=4cm]{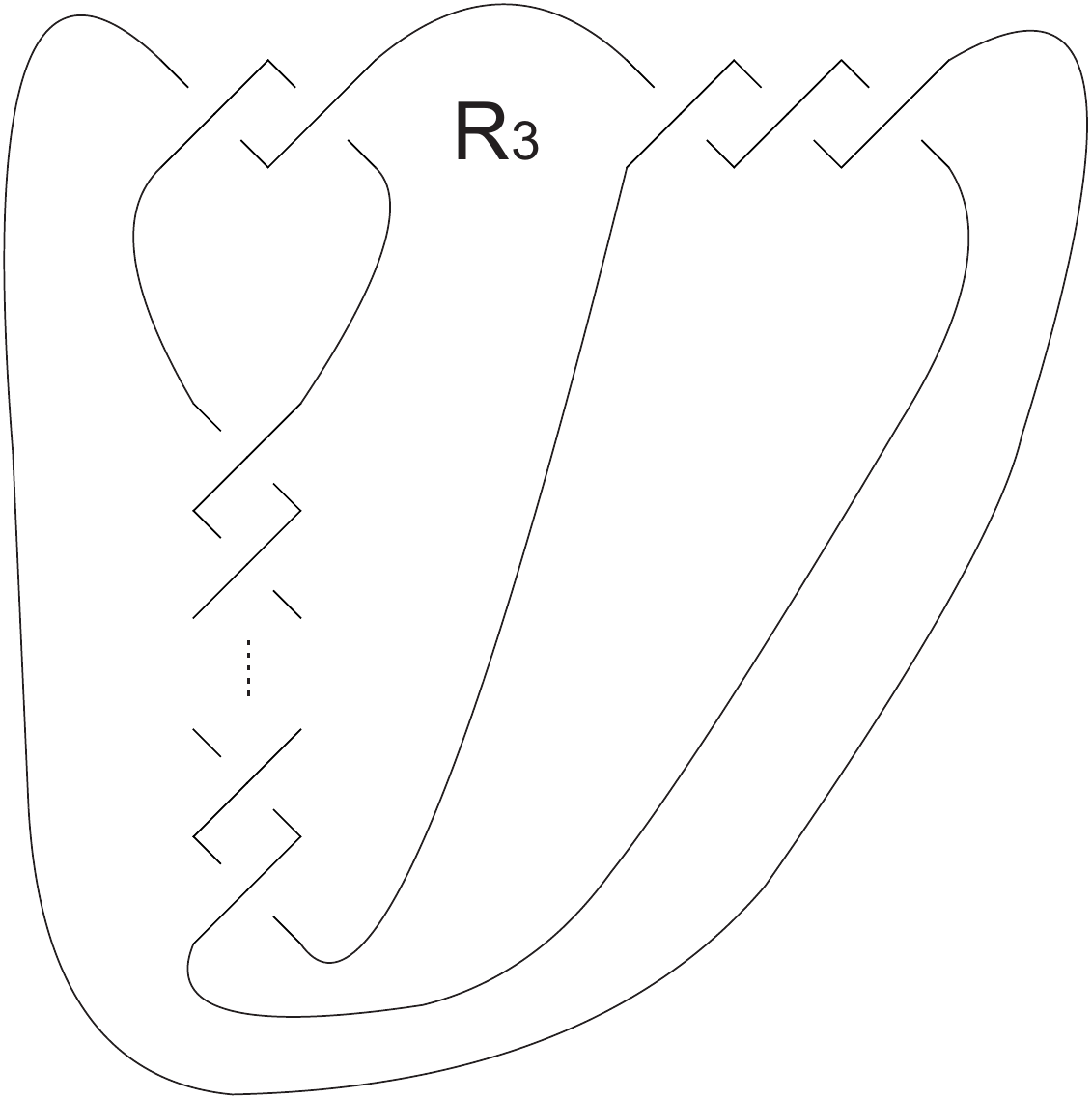}~~~~~~~~
\includegraphics[width=4cm,height=4cm]{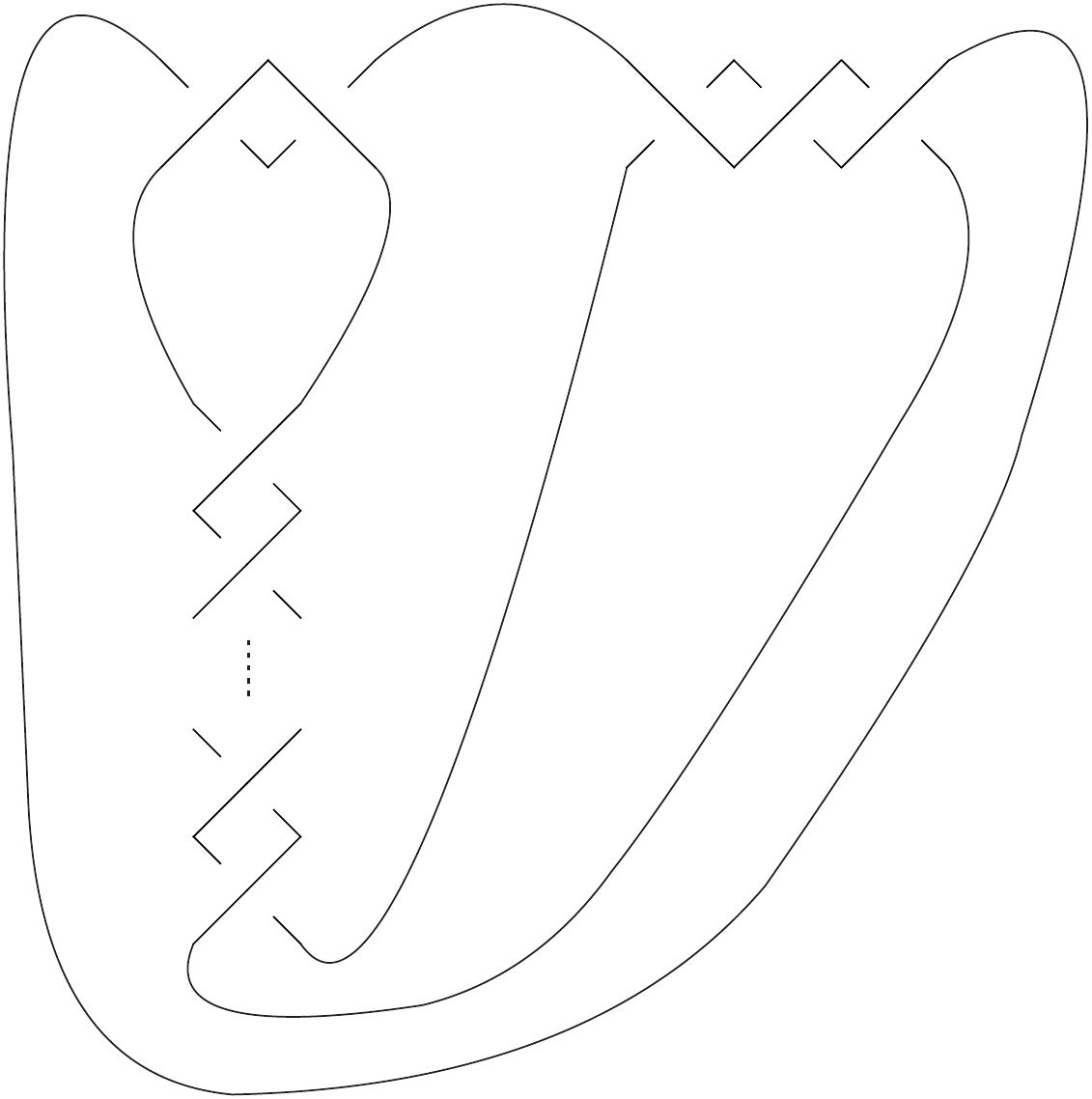}
\vspace*{8pt}
\caption{$C(2\ p\ 3)$}\label{2p3}
\end{center}
\end{figure}
\end{remark}

\begin{theorem}\label{5.thm29}
For 2-bridge knot/proper link $K$ with Conway's notation $C(m,\ p,\ n)$, where $p$ is even, we have
\begin{enumerate}
\item when $p \equiv 2\ (mod\ 4)$, \[u_R(K)\leq \left\lfloor\frac{|m-n|+2}{4}\right\rfloor+\left\lfloor\frac{p+2}{4}\right\rfloor\]
\item  when $p \equiv 0\ (mod\ 4)$,
\begin{align*}
  u_R(K)  \leq
  \begin{cases}
  \left\lfloor\frac{m+n+2}{4}\right\rfloor        & \text{if  either $m$ or $n$ is even }\\
   \frac{m+n+p}{4}                                & \text{if  both $m$ and $n$ are odd}
  \end{cases} .
 \end{align*}
\end{enumerate}
\end{theorem}

\begin{proof}
\noindent \textbf{\underline{Case (i)}} Observe that after region crossing changes at $\langle R_3', R_7', \cdots, R_{3 + 4\cdot \lfloor\frac{p-2}{4}\rfloor}'\rangle$ regions in $C(m,\ p,\ n)$, the resultant diagram is a diagram of $(2, m-n))$-type torus knot/link. Since region unknotting number for $(2,n-m)$ torus knot/link is  $\left\lfloor\frac{|m-n|+2}{4}\right\rfloor$, by selecting any non-consecutive $\left\lfloor\frac{|m-n|+2}{4}\right\rfloor$ regions in $t_m$ (if $m>n$) or in $t_n$ (if $n>m$) results in a trivial knot diagram. Thus \[ u_R(K)\leq \left\lfloor\frac{|m-n|+2}{4}\right\rfloor+\left\lfloor\frac{p+2}{4}\right\rfloor.\]

\noindent \textbf{\underline{Case (ii)}} If  either $m$ or $n$ is even, then proof directly follows from Theorem \ref{5.thm28}. When both $m$ and $n$ are odd, it is easy to observe that after region crossing changes at $\langle R_3', R_7', \cdots, R_{3 + 4\cdot \lfloor\frac{p-2}{4}\rfloor}'\rangle$ regions, the resultant diagram is a diagram of $(2,\ m+n))$-type torus knot/link. Since region unknotting number for $(2,\ m+n)$ torus knot/link is $\lfloor\frac{m+n+2}{4}\rfloor$, region crossing changes at $\langle R_3, R_7, \cdots, R_{3 + 4\cdot \lfloor\frac{m+n-2}{4}\rfloor}, R_3', R_7', \cdots, R_{3 + 4\cdot \lfloor\frac{p-2}{4}\rfloor}'\rangle$ in $C(m,\ p,\ n)$ transform $C(m,\ p,\ n)$ to a diagram of trivial knot. Observe that the number of region crossing changes is equal to $ \left\lfloor\frac{m+n+2}{4}\right\rfloor+\left\lfloor\frac{p+2}{4}\right\rfloor $. Since, in this case, $2$-bridge link $C(m,\ p,\ n)$ is proper iff $m+n\equiv 0\ (mod\ 4)$, we have $\lfloor\frac{m+n+2}{4}\rfloor = \frac{m+n}{4}$. Thus \[u_R(K)\leq  \frac{m+n+p}{4}.\] \hfill $\square$
\end{proof}

\noindent In $C(m,\ p,\ n)$, if $p$ is odd and either $m$ or $n$ is even, by Theorem \ref{5.thm28}, we have $u_R(K)\leq \left\lfloor\frac{m+n+2}{4}\right\rfloor$. In case when both $m$ and $n$ are odd, take $k = min\{m,\ n\}$. If either one of $m$ or $n$ is odd then consider $k$ to be that integer which is odd. In the following theorem, consider $k$ as defined above.

\begin{theorem}\label{5.thm31}
For 2-bridge knot $K$ with Conway's notation $C(m,\ p,\ n)$, where $p$ is odd and
\begin{enumerate}
\item if $p+k\equiv 0\ (mod\ 4)$, then \[u_R(K)\leq  \frac{p+k}{4}\]
\item if $p+k\equiv 2\ (mod\ 4)$, then \[u_R(K)\leq \frac{p+k+2}{4}.\]
\end{enumerate}
\end{theorem}

\begin{proof}
\noindent \textbf{\underline{Case (i)}}  If $p+k\equiv 0\ (mod\ 4)$, then based on $ k =m$ or $n$, we make region crossing changes at either $ \langle R_{ 3}, R_{ 7}, \cdots, R_{3 + 4\cdot \lfloor\frac{n-2}{4}\rfloor}, R_{3}', R_{7}', \cdots, R_{3 + 4\cdot \lfloor\frac{p-2}{4}\rfloor}'\rangle$ or $ \langle R_{m + 3}, R_{m + 7}, \cdots, R_{m + 3 + 4\cdot \lfloor\frac{n-2}{4}\rfloor},$ $R_{3}',$  $R_{7}',\\ \cdots, R_{3 + 4\cdot \lfloor\frac{p-2}{4}\rfloor}'\rangle$ regions. After these region crossing changes we get a diagram of trivial knot. In any case, the resultant diagram is equivalent to diagram of $C(\pm1,\ \mp1,\ n) = C(0,\ n)$ or $C(m,\ \pm1,\ \mp1) = C(m,\ 0)$, which is a trivial knot. Observe that the number of region crossing changes is equal to $ \lfloor\frac{k+2}{4}\rfloor + \lfloor\frac{p+2}{4}\rfloor  = \frac{p+k}{4}$ and hence, $ u_R(K)\leq  \frac{p+k}{4}. $\\

\noindent \textbf{\underline{Case (ii)}}  If $p+k\equiv 2\ (mod\ 4)$, then we have either both $p$ and $k$ are $\equiv 1\ (mod\ 4)$ or both $p$ and $k$ are $\equiv 3\ (mod\ 4)$. If both $p$ and $k$ are $\equiv 1\ (mod\ 4)$, then based on $k = m$ or $n$, we make region crossing changes at $ \langle R_{3}, R_{7}, \cdots, R_{3 + 4\cdot \lfloor\frac{n-2}{4}\rfloor},$ $ R_{3}', R_{7}', \cdots, R_{3 + 4\cdot \lfloor\frac{p-2}{4}\rfloor}'\rangle$ or $ \langle R_{m + 3}, R_{m + 7}, \cdots, R_{m + 3 + 4\cdot \lfloor\frac{n-2}{4}\rfloor},$ $ R_{3}', R_{7}', \cdots, R_{3 + 4\cdot \lfloor\frac{p-2}{4}\rfloor}'\rangle$ regions. In case if  both $p$ and $k$ are $\equiv 3\ (mod\ 4)$, then based on $k = m$ or $n$, we make region crossing changes at $ \langle R_{ 3}, R_{7}, \cdots, R_{3 + 4\cdot \lfloor\frac{n-2}{4}\rfloor},$ $ R_{3}', R_{7}', \cdots, R_{3 + 4\cdot \lfloor\frac{p-4}{4}\rfloor}'\rangle$ or $ \langle R_{m + 3}, R_{m + 7}, \cdots, R_{m + 3 + 4\cdot \lfloor\frac{n-2}{4}\rfloor},$ $ R_{3}', R_{7}', \cdots,$ $R_{3 + 4\cdot \lfloor\frac{p-4}{4}\rfloor}'\rangle$ regions. In case if $p = 3$, then we make no region region crossing changes in $t_p'$. After these region crossing changes, we get a diagram of a 2-bridge knot $C(m,\ 2)$. Then one more region crossing change at $R_1'$ transforms $C(m,\ p,\ n)$ to an unknot. Observe that the number of region crossing changes is equal to $ \lfloor\frac{k+2}{4}\rfloor+\lfloor\frac{p+2}{4}\rfloor+1 $ or $ \lfloor\frac{k+2}{4}\rfloor+\lfloor\frac{p}{4}\rfloor+1 $ respectively. Thus $ u_R(K)\leq \frac{p+k+2}{4}.$ \hfill $\square$
\end{proof}

\begin{remark}\label{5.rmk4}
Note that in case of those 2-bridge knots which occurs in more than one category, we consider upper bound for region unknotting number to be minimum of the upper bounds from all undertaken categories. For example 2-bridge knot $C(8,\ 5,\ 3)$ satisfies the hypothesis of Theorem \ref{5.thm28} and Theorem \ref{5.thm31}. By Theorem \ref{5.thm28}, $ u_R(K)\leq 3 $, as region crossing changes at $R_3$, $R_7$ and $R_{11}$ regions in Figure \ref{853}, transforms $C(8,\ 5,\ 3)$ into trivial knot and by Theorem \ref{5.thm31}, $ u_R(K)\leq 2 $, as region crossing changes at $R_{11}$ and $R'_3$ makes it unknot. So $u_R(C(8,\ 5,\ 3)) \leq 2$.
\end{remark}

\begin{figure}
\begin{center}
\includegraphics[width=5cm,height=4cm]{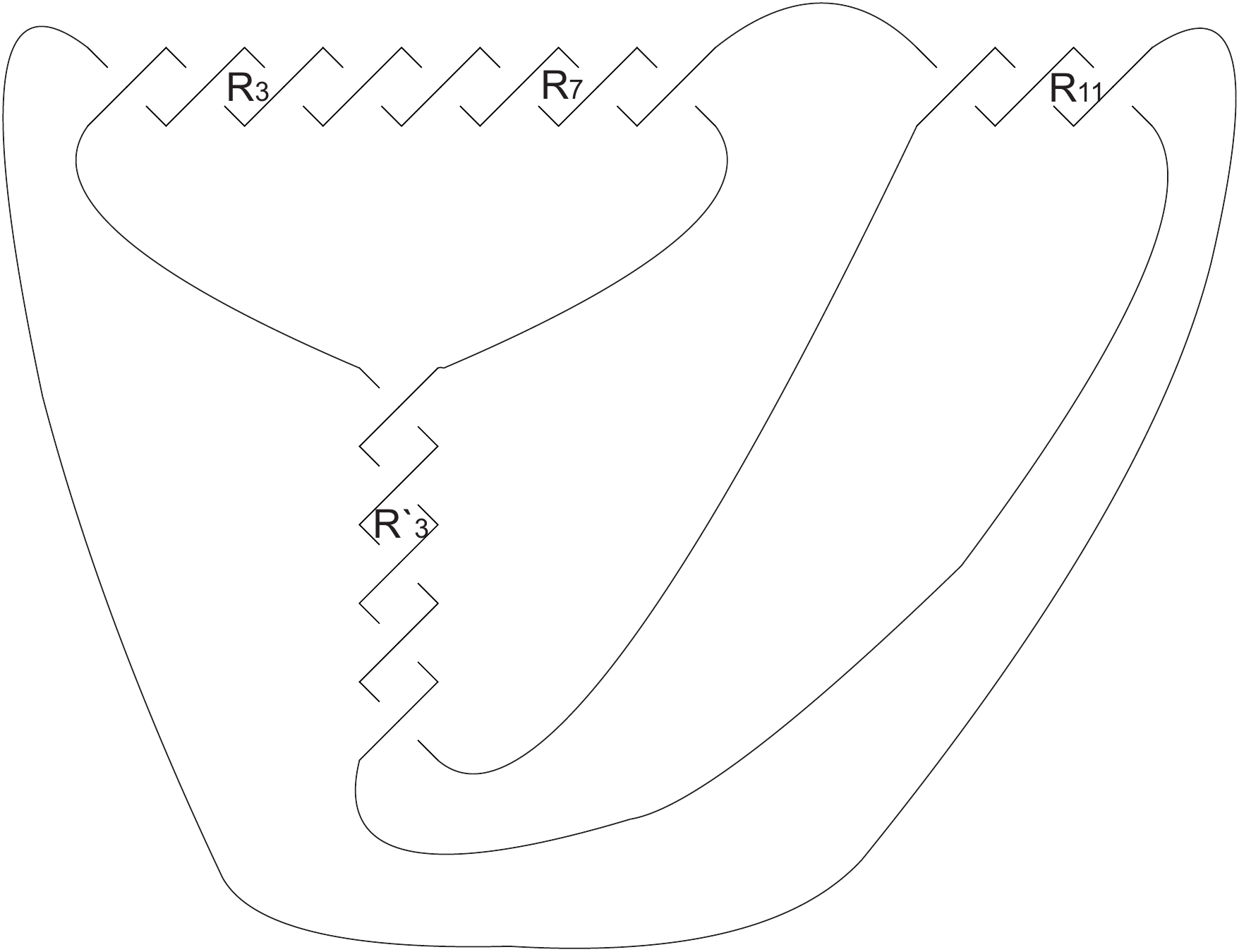}
\vspace*{8pt}
\caption{Region unknotting number for $C(8,\ 5,\ 3)$}\label{853}
\end{center}
\end{figure}

\begin{figure}
\begin{center}
\includegraphics[width=6cm,height=1cm]{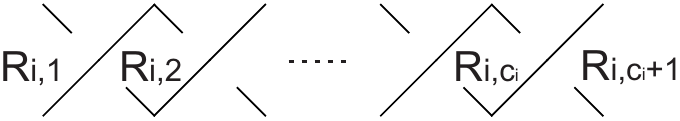}
\vspace*{8pt}
\caption{Region Data for integer tangle $t_{c_i}$}\label{integerdata}
\end{center}
\end{figure}

\noindent A general upper bound for region unknotting number is given for all $2$-bridge knot/proper links. Due to generality, here we consider regions of each integer tangle ($t_{c_i}$ or $t_{c_i}'$) as $R_{i,1},\ R_{i,2},\ \cdots,\ R_{i,c_i+1}$ as in Figure \ref{integerdata}. Note that region $R_{i,c_i+1}$ of tangle $t_{c_i}$  is same as the region $R_{i+2,1}$ of $t_{c_{i+2}}$. Similarly the region $R_{i,c_i+1}$ of tangle $t_{c_i}'$ is same as the region $R_{i+2,1}$ of $t_{c_{i+2}}'$ respectively.\\

\noindent To provide an upper bound for $C(c_1,\ c_2, \cdots, c_n)$, first we construct a subset $L$ of $2\mathbb{N}$ as follows:
\begin{itemize}
\item for $j = 2$, if $c_1\equiv 0\  (mod\ 2)$ then $2 \in L$.
\item for next even integer $j = 4$, if $2\not\in L$ and $c_1+c_2+c_3 \equiv 0\ (mod\ 2)$ then $4 \in L$. If $2\in L$ and $c_1+c_3 \equiv 0\ (mod\ 2)$ then $4 \in L$.
\item Continuing in the same way, any even integer $j(\leq n)\ \in L$ if \[\sum_{\substack{i<j \\ i\not\in L}} c_i \equiv 0\ (mod\ 2).\]
\end{itemize}

\noindent In the following theorem, we will observe that by selectively choosing some region crossing changes, there is no need to make any region crossing change in $t_{c_j}'$, where $j\in L$, to transform $C(c_1,\ c_2, \cdots, c_n)$ to a trivial knot.
\begin{theorem}\label{5.thm24}
For 2-bridge knot/proper link $K$ with Conway's notation $C(c_1,\ c_2, \cdots, c_n)$,
\[u_R(K)\leq \left\lfloor\frac{\sum\limits_{i\not\in L} c_i\ +2}{4}\right\rfloor.\]
\end{theorem}
\begin{proof}
Since any integer is either $\equiv -2,\ -1,\ 0$ or $1\ (mod\ 4)$, we can say that $ \sum_{\substack{i<j \\ i\not\in L}} c_i \equiv k_j\ (mod\ 4) $ where for each $j$ ($1\leq j \leq n$), $k_j = -2,\ -1,\ 0$ or $1$. Note that for each $j\in L$, after making region crossing changes at $R_{i,3-k_i}; R_{i,7-k_i}; \cdots; R_{i,3+4\cdot \left\lfloor\frac{c_i\ -2}{4}\right\rfloor-k_i}$ regions in each $c_i$ (for $i<j $, $i\not\in L$ and $c_i+k_i\geq 2$), $c_j$ can be untangle by just simple twists. Hence for each $j \not\in L$, if $c_j+k_j\geq 2$ then region crossing changes at $R_{j,3-k_j}; R_{j,7-k_j}; \cdots; R_{j,3+4\cdot \left\lfloor\frac{c_j\ -2}{4}\right\rfloor-k_j}$ regions in $c_j$ of $C(c_1,\ c_2, \cdots, c_n)$ results a diagram of a trivial knot. \hfill $\square$
\end{proof}

\noindent Suppose the $2$-bridge knot with Conway's notation $C(c_1,\ c_2, \cdots, c_n)$ satisfies
\begin{enumerate}
\item $c_{2k+1} \text{ is } even$ for each non-negative integer $k$ such that $2k+1 \leq n$ or
\item $c_{2k} \text{ is } even$ for each positive integer $k$ such that $2k \leq n$ and $n$ is even
\end{enumerate}
then we can provide a better upper bound for region unknotting number. Note that a $2$-bridge link $C(c_1,\ c_2, \cdots, c_n)$, where $c_{2k+1} = even$ for each non-negative integer $k$ such that $2k+1 \leq n$ and $n$ is odd, is not proper iff $\sum\limits_{i=2k+1} c_i \equiv 2\ (mod\ 4)$.

\begin{theorem}\label{5.thm25}
For a 2-bridge knot $C(c_1,\ c_2, \cdots, c_n)$
\begin{enumerate}
\item if $c_{2k+1} \text{ is } even$ for each non-negative integer $k$ such that $2k+1 \leq n$ then
\[\displaystyle{ u_R(K)\leq \left\lfloor\frac{\sum\limits_{i=2k+1} c_i\ + 2}{4}\right\rfloor},\]
\item if $c_{2k} \text{ is } even$ for each positive integer $k$ such that $2k \leq n$ and $n$ is even then
\[\displaystyle{ u_R(K)\leq min\left\{\left\lfloor\frac{\sum\limits_{i=2k} c_i\ + 2}{4}\right\rfloor, \left\lfloor\frac{\sum\limits_{i\not\in L} c_i\ +2}{4}\right\rfloor\right\}}.\]
\end{enumerate}
\end{theorem}
\begin{proof}

\noindent \textbf{\underline{Case (i)}} When $c_{2k+1} \text{ is } even$ for each non-negative integer $k$ such that $2k+1 \leq n$.\\

\noindent Consider the $2$-bridge knot $C(c_1,\ c_2, \cdots, c_n)$, where $c_{2k+1} \text{ is } even$. Region crossing changes at $\langle R_3, R_7, \cdots,$ $R_{3+4\cdot \left\lfloor\frac{{\sum\limits_{i=2k+1}c_i}\ -2}{4}\right\rfloor}\rangle$ regions in $C(c_1,\ c_2, \cdots, c_n)$, as in Figure \ref{regiondata2-bridge}, results in a diagram of trivial knot. Observe that the region crossing change at $R_j$ for any  $j \in \{c_1+1,\ c_1+c_2+1,\cdots,\sum\limits_{i=1}^{\lceil\frac{n-2}{2}\rceil}c_{2i-1}+1\}$, results in $2$ crossing changes in horizontal tangles, one in some $t_{c_{2k-1}}$ and other in $t_{c_{2k+1}}$. Region crossing change at other region $R_i$ will result in $2$ crossing changes in some $t_{c_{2k-1}}$. After making above said region crossing changes, the absolute value of sum of signs of all the crossings of horizontal tangles in the resultant diagram becomes zero. Hence \[\displaystyle{ u_R(K)\leq \left\lfloor\frac{\sum\limits_{i=2k+1} c_i\ + 2}{4}\right\rfloor}.\]

\noindent \textbf{\underline{Case (ii)}} Proof follows similarly as in Case (i). Here we need to change $\langle R_3', R_7', \cdots, R_{3 + 4\cdot \left\lfloor\frac{\sum\limits_{i=2k}c_i\ \ -2}{4}\right\rfloor}'\rangle$ regions to get a diagram of trivial knot. \hfill $\square$
\end{proof}

\begin{example}
Consider the $2$-bridge knot $C(2,\ 3,\ 4,\ 2,\ 6)$. Observe that region crossing changes at $\langle R_3, R_7, R_{11}\rangle$ regions transform $C(2,\ 3,\ 4,\ 2,\ 6)$ to a diagram of trivial knot as shown in Figure \ref{23426}. Figure \ref{23426}(a) is showing $2$-bridge knot $C(2,\ 3,\ 4,\ 2,\ 6)$ and the diagram of trivial knot obtained by above said region crossing changes is shown in Figure \ref{23426}(b).

\begin{figure}
\begin{center}
\begin{subfigure}[b]{0.4\textwidth}
                \includegraphics[width=\textwidth]{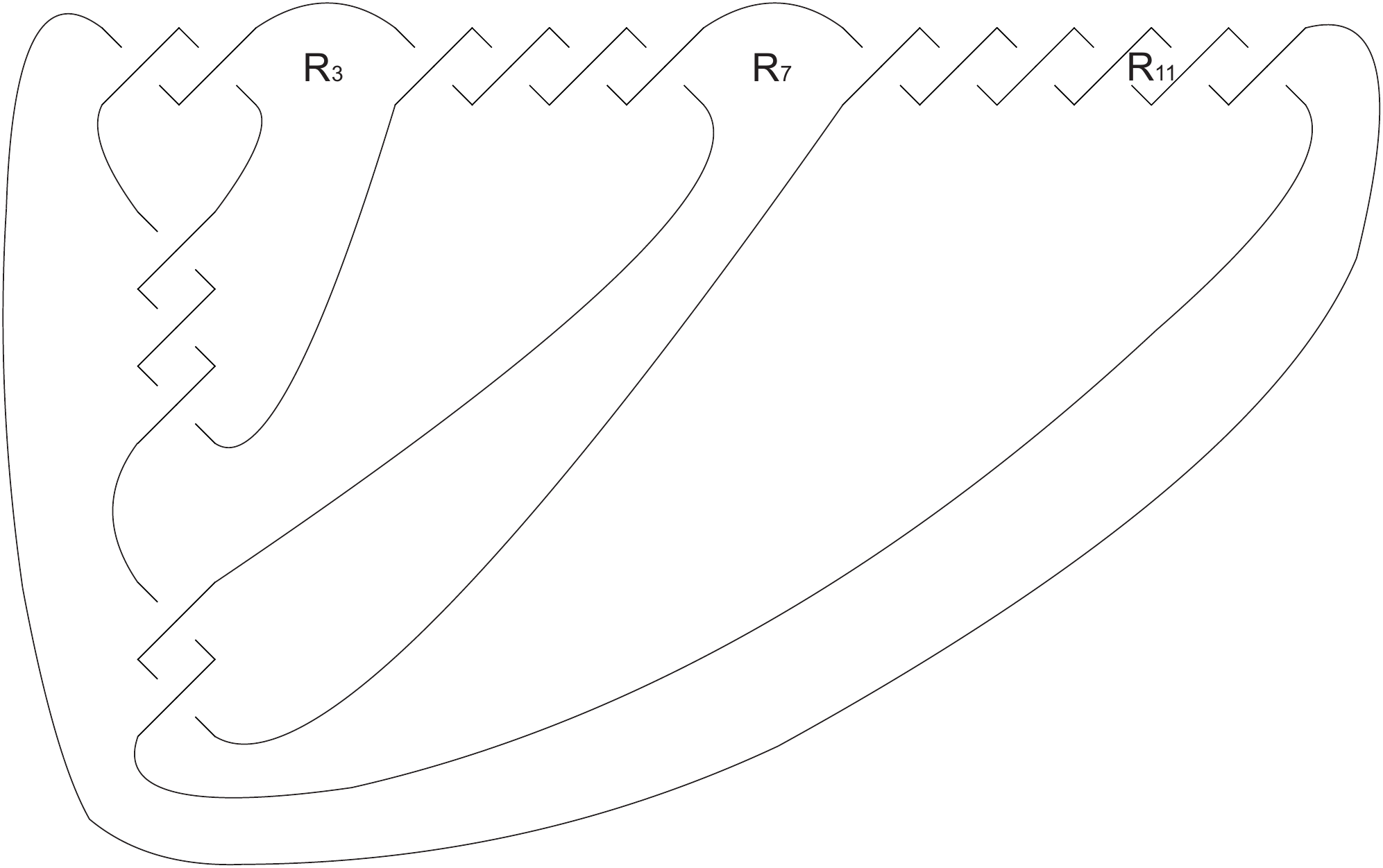}
                \caption{}
                \label{23426ex}
        \end{subfigure}%
        ~ %add desired spacing between images, e. g. ~, \quad, \qquad, \hfill etc.
          %(or a blank line to force the subfigure onto a new line)
%        \begin{subfigure}[b]{0.3\textwidth}
%                \includegraphics[width=\textwidth]{33orientation1}
%                \caption{A tiger}
%                \label{fig:tiger}
%        \end{subfigure}
        ~ %add desired spacing between images, e. g. ~, \quad, \qquad, \hfill etc.
          %(or a blank line to force the subfigure onto a new line)
        \begin{subfigure}[b]{0.4\textwidth}
                \includegraphics[width=\textwidth]{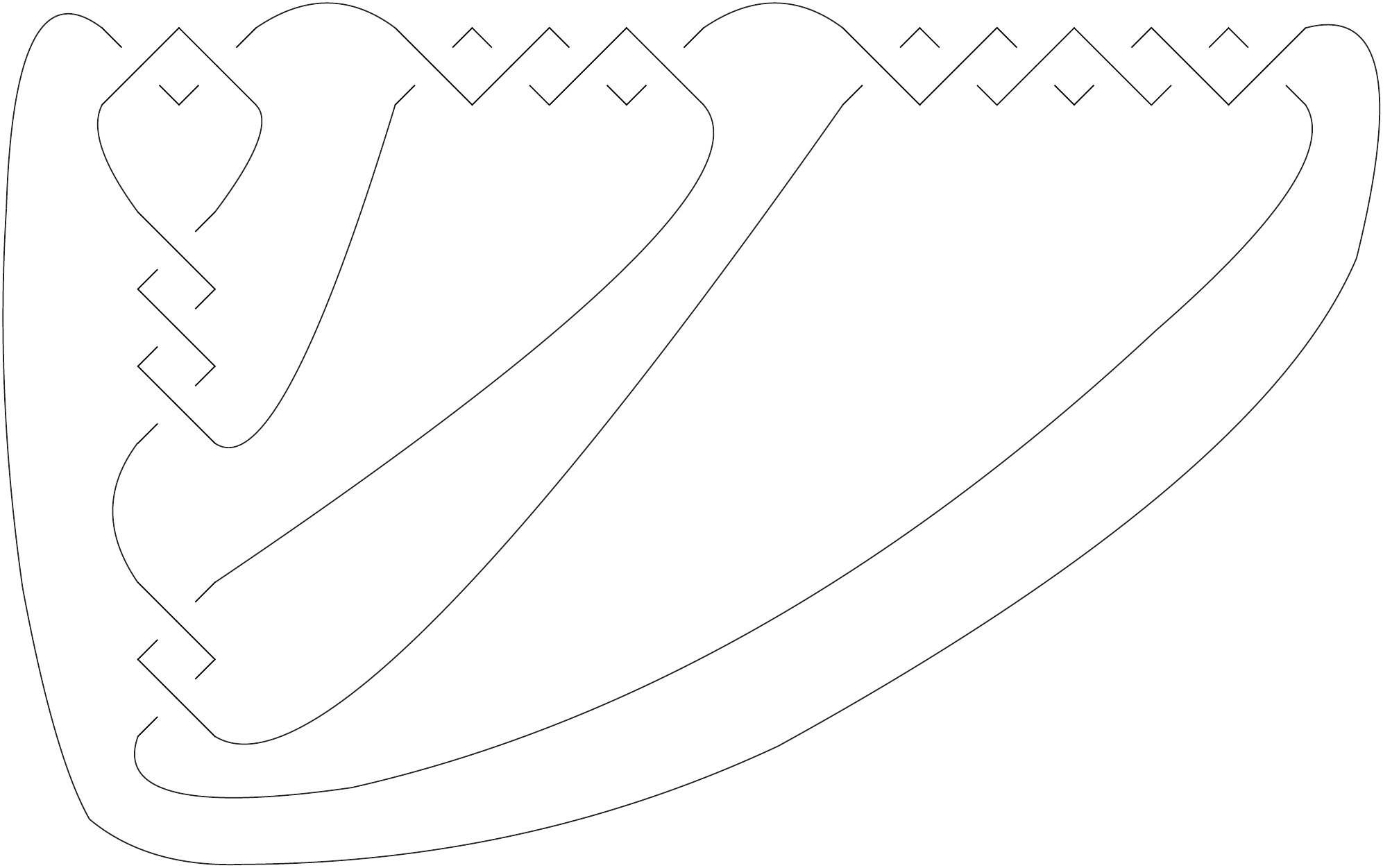}
                \caption{}
                \label{23426ex1}
        \end{subfigure}
\vspace*{8pt}
\caption{$2$-bridge link $C(2,\ 3,\ 4,\ 2,\ 6)$}\label{23426}
\end{center}
\end{figure}
\end{example}

\noindent Based on the above results, region unknotting number for some $2$-bridge knots is provided in Table~\ref{table1}.

\begin{table}[ht]%%%%%%%%%%%%%%%%%%%%%%%%%%%%%%1%%%%%%%%%%%%%%%%%%%%%
\begin{center}
\caption{Region unknotting number of some $2$-bridge knots}%Summary of results obetained}
\begin{tabular}{| c| c | c| c | l |}
\hline
Knot $K$ & $u_R(K)$ & Knot $K$ & $u_R(K)$\\
%H-T Notation  &  Number & Knot & Crossing Data &\\
\hline
$C(2,\ c_2,\ 2,\ c_4)$   &  1 &     $C(c_1,\ 2,\ c_3,\ 2)$     &   1      \\
 $C(6,\ c_2,\ 2)$  &  $\leq 2$ &   $C(4,\ c_2,\ 2,\ c_4,\ 2)$        &    $\leq 2$     \\
  $C(4,\ c_2,\ 4)$  &  $\leq 2$ &  $C(2,\ c_2,\ 4,\ c_4,\ 2)$     &    $\leq 2$     \\
  $C(2,\ c_2,\ 6)$  &  $\leq 2$ &  $C(2,\ c_2,\ 2,\ c_4,\ 4)$       &   $\leq 2$      \\
   $C(4,\ c_2,\ 2)$  &  $\leq 2$ &   $C(2,\ c_2,\ 2,\ c_4,\ 2)$      &  $\leq 2$       \\
    $C(2,\ c_2,\ 4)$  &  $\leq 2$ &    $C(4,\ c_2,\ 2,\ c_4,\ 2,\ c_6)$      &  $\leq 2$       \\
 $C(6,\ c_2,\ 2,\ c_4)$ &  $\leq 2$ &    $C(2,\ c_2,\ 4,\ c_4,\ 2,\ c_6)$      &   $\leq 2$      \\
 $C(4,\ c_2,\ 4,\ c_4)$    &  $\leq 2$ &    $C(2,\ c_2,\ 2,\ c_4,\ 4,\ c_6)$     &  $\leq 2$       \\
 $C(2,\ c_2,\ 6,\ c_4)$  &  $\leq 2$ &  $C(2,\ c_2,\ 2,\ c_4,\ 2,\ c_6)$      &   $\leq 2$      \\
 $C(4,\ c_2,\ 2,\ c_4)$  & $\leq 2$ & $C(2,\ c_2,\ 2,\ c_4,\ 2,\ c_6,\ 2)$       &  $\leq 2$       \\
 $C(2,\ c_2,\ 4,\ c_4)$   &  $\leq 2$ &   $C(2,\ c_2,\ 2,\ c_4,\ 2,\ c_6,\ 2,\ c_8)$      &  $\leq 2$       \\
\hline
\end{tabular} \label{table1}
\end{center}
\end{table}

%%%%%%%%%%%%%%%%%%%%%%%%%%%%%%%%%%%%%%%%%%%%%%%%%%%%%%%%%%%%%%%%%%%%%%%
%%%%%%%%%%%%%%%%%%%%%%%%%%%%%%%%SECTION-3%%%%%%%%%%%%%%%%%%%%%%%%%%%%%%
%%%%%%%%%%%%%%%%%%%%%%%%%%%%%%%%%%%%%%%%%%%%%%%%%%%%%%%%%%%%%%%%%%%%%%%
\section{Arf Invariant of $2$-bridge knots}\label{arfinvariant}
\noindent In this section we discuss the Arf invariant of $2$-bridge knots. Arf invariant is well defined for knots and proper links in different ways \cite{properk, arf, onknots}. A relation between Arf invariant and region crossing change is given by Z. Cheng \cite{proper}. Let $L$ be a reduced diagram of a proper link and $L'$ is obtained from $L$ by a region crossing change at a region $R$ in $L$ then $L'$ is also proper and there Arf invariants are related by the following theorem:

\begin{theorem} \cite{proper}
Let $L$ be a diagram of a proper link, $L'$ is obtained by taking region crossing change on region $R$ of $L$, where $R$ is white colored region in checkerboard coloring of $L$, then
\begin{center}
\emph{Arf}$(L)+$\emph{Arf}$(L')=$
$\begin{cases}
0$ $ ($\emph{mod 2}$)& \text{\emph{if} $\frac{1}{2}\sum\limits_{i=1}^m(a(c_i)-w(c_i))\equiv 0$ $($\emph{mod}$ $ $4 )$;}\\
1$ $ ($\emph{mod 2}$)& \text{\emph{if} $\frac{1}{2}\sum\limits_{i=1}^m(a(c_i)-w(c_i))\equiv 2$ $($\emph{mod}$ $ $4 )$.}
\end{cases}$
\end{center}
Here $\{c_1,\ c_2, \cdots, c_m\}$ denote the crossing points on the boundary of $R$ and $a(c_i)$ and $w(c_i)$ are defined as in Figure \ref{aciwci}. \hfill $\square$
\end{theorem}

\begin{figure}
\begin{center}
\includegraphics[width=11cm,height=2cm]{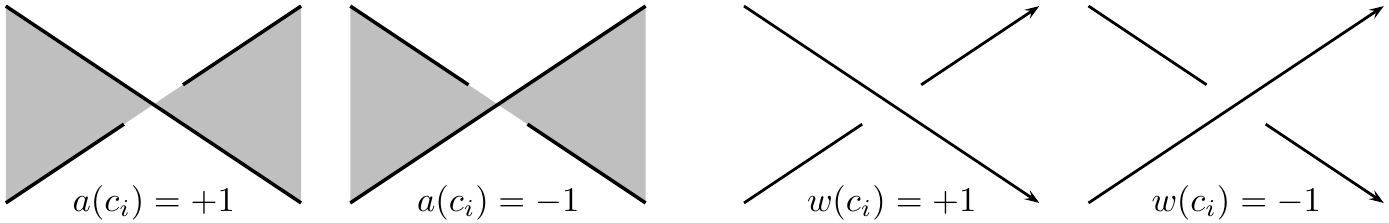}
\vspace*{8pt}
\caption{$a(c_i)$ and $w(c_i)$}\label{aciwci}
\end{center}
\end{figure}

\noindent Also we can use this to calculate Arf invariant of a knot or proper link. Assuming unbounded region as white colored, if we denote

\begin{align*}
 A(R) =
  \begin{cases}
  \frac{1}{2}\sum\limits_{i=1}^m(a(c_i)-w(c_i))        & \text{if  $R$ is white colored };\\
   -\frac{1}{2}\sum\limits_{i=1}^m(a(c_i)+w(c_i))        & \text{if  $R$ is black colored}
  \end{cases}
 \end{align*}
then Z. Cheng proved the following in \cite{proper}:

\begin{theorem}\cite{proper}\label{5.thm.3.2}
Let $L$ be a reduced diagram of a proper link, $R_1, \cdots, R_n$ some regions of $L$, such that region crossing changes at $R_1, \cdots, R_n$ will turn $L$ to be trivial. Then,
\begin{center}
\emph{Arf}$(L)=$
$\begin{cases}
0& \text{\emph{if} $\sum\limits_{i=1}^nA(R_i)\equiv 0$ $($\emph{{mod} 4}$ )$;}\\
1& \text{\emph{if} $\sum\limits_{i=1}^nA(R_i)\equiv 2$ $($\emph{{mod} 4}$ )$.}
\end{cases}$
\end{center} \hfill $\square$
\end{theorem}

\begin{figure}
\begin{center}
\includegraphics[width=5cm,height=1cm]{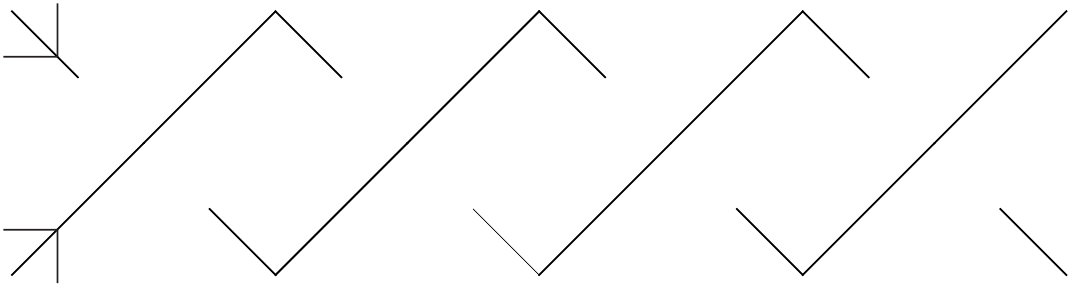} ~~~~~~~~~~~~
\includegraphics[width=5cm,height=1cm]{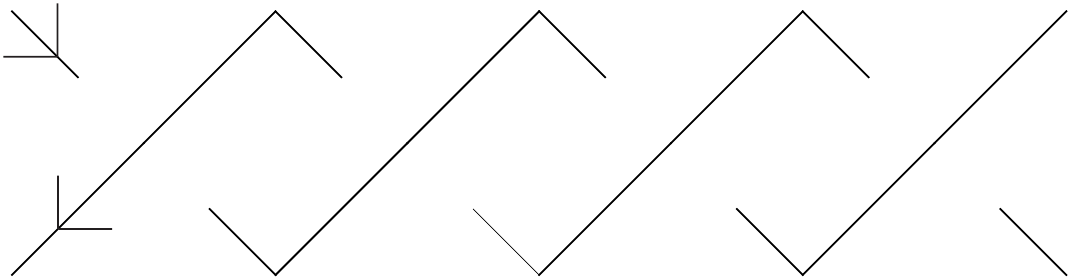}
\vspace*{8pt}
\caption{Parallel and opposite oriented tangles}\label{figparopp}
\end{center}
\end{figure}

\noindent We use Theorem \ref{5.thm.3.2} to calculate Arf invariant of different $2$-bridge knots. In a $2$-bridge knot $C(c_1,\ c_2, \cdots, c_n)$, we call a $2$-tangle a parallel oriented tangle or an opposite oriented tangle based on whether both the strands are parallel or opposite oriented as in Figure \ref{figparopp}. Observe that for any crossing $c$ in horizontal tangle $t_{c_i}$, $w(c) = -1$ or $1$ based on whether $t_{c_i}$ is parallel or opposite oriented. Similarly for any crossing $c$ in vertical tangle $t_{c_i}'$, $w(c) = 1$ or $-1$ based on whether  $t_{c_i}'$ is parallel or opposite oriented. The values of $w(c)$ are shown in Table \ref{table2}. Also $a(c)$ for any crossing $c$ in $2$-bridge knot is $-1$.\\

\begin{table}[ht]%%%%%%%%%%%%%%%%%%%%%%%%%%%%%%1%%%%%%%%%%%%%%%%%%%%%
\begin{center}
\caption{Value of $w(c)$ in different conditions}%Summary of results obetained}
\begin{tabular}{| c| c | c| c | l |}
\hline
 & parallel oriented & opposite oriented \\
\hline
horizontal tangle ($t_{c_i}$)   &  -1 &   1       \\
vertical tangle ($t_{c_i}'$)    &   1 &  -1      \\
\hline
\end{tabular} \label{table2}
\end{center}
\end{table}

\noindent Note that Arf invariant of a link depends on orientation. Let $L$ be an oriented link and link $L'$ is obtained from $L$ after changing the orientation of one of its component. Observe that Arf invariants of $L$ and $L'$ may differ since smoothing of crossings according to orientation in $L$ and $L'$ may result with knots having different Arf invariants. Dependance of Arf invariant on the orientation of links is shown in Example \ref{5.ex3}. However, Arf invariant of a knot remains same on reversing the orientation of knot. We calculate Arf invariant for some $2$-bridge knot (not link) classes assuming unbounded region as white colored as in Figure \ref{regiondata2-bridge}. Observe that $C(m,\ n)$ is a link diagram if both $m$ and $n$ are odd. In the following theorem, we calculate Arf invariant for $2$-bridge knot (not link) whose Conway's notation is $C(m,\ n)$.

\begin{theorem}
Consider a $2$-bridge knot with Conway's notation $C(m,\ n)$
\begin{enumerate}
\item when only one of either $m$ or $n$ is even. Specifically, if $n$ is even and
\begin{enumerate}
\item  $n\equiv \ 0\ (mod\ 4)$, then
\begin{center}
\emph{Arf}$(C(m,\ n))=$
$\begin{cases}
0& \text{\emph{if} $\frac{n}{2} \equiv 0$ $($\emph{{mod} 4}$ )$;}\\
1& \text{\emph{if} $\frac{n}{2}\equiv 2$ $($\emph{{mod} 4}$ )$.}
\end{cases}$
\end{center}
\item  $n\equiv \ 2\ (mod\ 4)$, then
\begin{center}
\emph{Arf}$(C(m,\ n))=$
$\begin{cases}
0& \text{\emph{if} $m+\frac{n}{2} \equiv 0$ $($\emph{{mod} 4}$ )$;}\\
1& \text{\emph{if} $m+\frac{n}{2}\equiv 2$ $($\emph{{mod} 4}$ )$.}
\end{cases}$
\end{center}
\end{enumerate}

\item when both $m$ and $n$ are even and if
\begin{enumerate}
\item  $n\equiv \ 0\ (mod\ 4)$, then
\begin{center}
\emph{Arf}$(C(m,\ n))= 0$
\end{center}
\item  $n\equiv \ 2\ (mod\ 4)$, then
\begin{center}
\emph{Arf}$(C(m,\ n))=$
$\begin{cases}
0& \text{\emph{if} $m \equiv 0$ $($\emph{{mod} 4}$ )$;}\\
1& \text{\emph{if} $m \equiv 2$ $($\emph{{mod} 4}$ )$.}
\end{cases}$
\end{center}
\end{enumerate}
\end{enumerate}
\end{theorem}

\begin{proof}
\noindent \textbf{\underline{Case (i)}} Observe that in $C(m,\ n)$, the horizontal tangle $t_m$ is opposite oriented and the vertical tangle $t_n'$ is parallel oriented. Hence at each crossing $c$, we have $w(c) = 1$ and $a(c) = -1$. Since the positions of regions $R_i$ which turn $C(m,\ n)$ to unknot, is provided in Theorem \ref{5.thm23}, it is easy to calculate $A(R_i)$ for each of these regions $R_i$. In particular for any $R_i$, \[A(R_i) = -(\text{number of crossings on}\ \partial R_i).\]
Hence
\begin{center}
$\sum A(R_i) =$
$\begin{cases}
-\frac{n}{2}& \text{\emph{if} $n \equiv 0$ $($\emph{{mod} 4}$ )$;}\\
-(m+\frac{n}{2})& \text{\emph{if} $n \equiv 2$ $($\emph{{mod} 4}$ )$.}
\end{cases}$
\end{center}
Since $\sum\limits_{i}A(R_i)$ is always even, by Theorem \ref{5.thm.3.2}, the result follows.\\

\noindent \textbf{\underline{Case (ii)}}  When both $m$ and $n$ are even then both $t_m$ and $t_n'$ tangles are opposite oriented. Hence for each crossing $c$ in $t_m$, $w(c) = 1$ and for each crossing $c$ in $t_n'$, $w(c) = -1$. Observe that all the regions $R_i$ provided by Theorem \ref{5.thm23}, whose change transform $C(m,\ n)$ into trivial knot, are white colored. Hence,
\begin{center}
$\sum A(R_i) =$
$\begin{cases}
0& \text{\emph{if} $n \equiv 0$ $($\emph{{mod} 4}$ )$;}\\
-m& \text{\emph{if} $n \equiv 2$ $($\emph{{mod} 4}$ )$.}
\end{cases}$
\end{center}
Now by Theorem \ref{5.thm.3.2}, the result follows.
\hfill $\square$
\end{proof}

\begin{example}\label{5.ex3}
Arf invariant of $2$-bridge proper link $C(m,\ n)$ shown in Figure \ref{mnlink}(a) and \ref{mnlink}(b) is not same. For Figure \ref{mnlink}(a)
\begin{center}
\emph{Arf}$(C(m,\ n))=$
$\begin{cases}
0& \text{\emph{if} $2\lfloor\frac{m+2}{4}\rfloor \equiv 0$ $($\emph{{mod} 4}$ )$;}\\
1& \text{\emph{if} $2\lfloor\frac{m+2}{4}\rfloor \equiv 2$ $($\emph{{mod} 4}$ )$;}
\end{cases}$
\end{center}
and for Figure \ref{mnlink}(b)
\begin{center}
\emph{Arf}$(C(m,\ n))=$
$\begin{cases}
0& \text{\emph{if} $2\lfloor\frac{n+2}{4}\rfloor \equiv 0$ $($\emph{{mod} 4}$ )$;}\\
1& \text{\emph{if} $2\lfloor\frac{n+2}{4}\rfloor \equiv 2$ $($\emph{{mod} 4}$ )$;}
\end{cases}$
\end{center}
respectively. Note that in case of Figure \ref{mnlink}(a), $a(c) = -1 = w(c)$ for each crossing $c$ and in case of Figure \ref{mnlink}(b), $a(c) = -1 = -w(c)$. Also $\sum\limits_{i=1}^kA(R_i) =
2\lfloor\frac{m+2}{4}\rfloor$ and $-2\lfloor\frac{n+2}{4}\rfloor$ for links in Figure \ref{mnlink}(a) and \ref{mnlink}(b), respectively. By Theorem \ref{5.thm23}, let region crossing change at $R_i \ (i = 1,2, \cdots, k)$  transforms $C(m,\ n)$ to a trivial link. For example the Arf invariant for $2$-bridge link $C(7,\ 5)$ is $0$ or $1$ when orientation is taken as in Figure \ref{mnlink}(a) or \ref{mnlink}(b) respectively.
\end{example}

\begin{figure}
\begin{center}
\begin{subfigure}[b]{0.2\textwidth}
                \includegraphics[width=\textwidth]{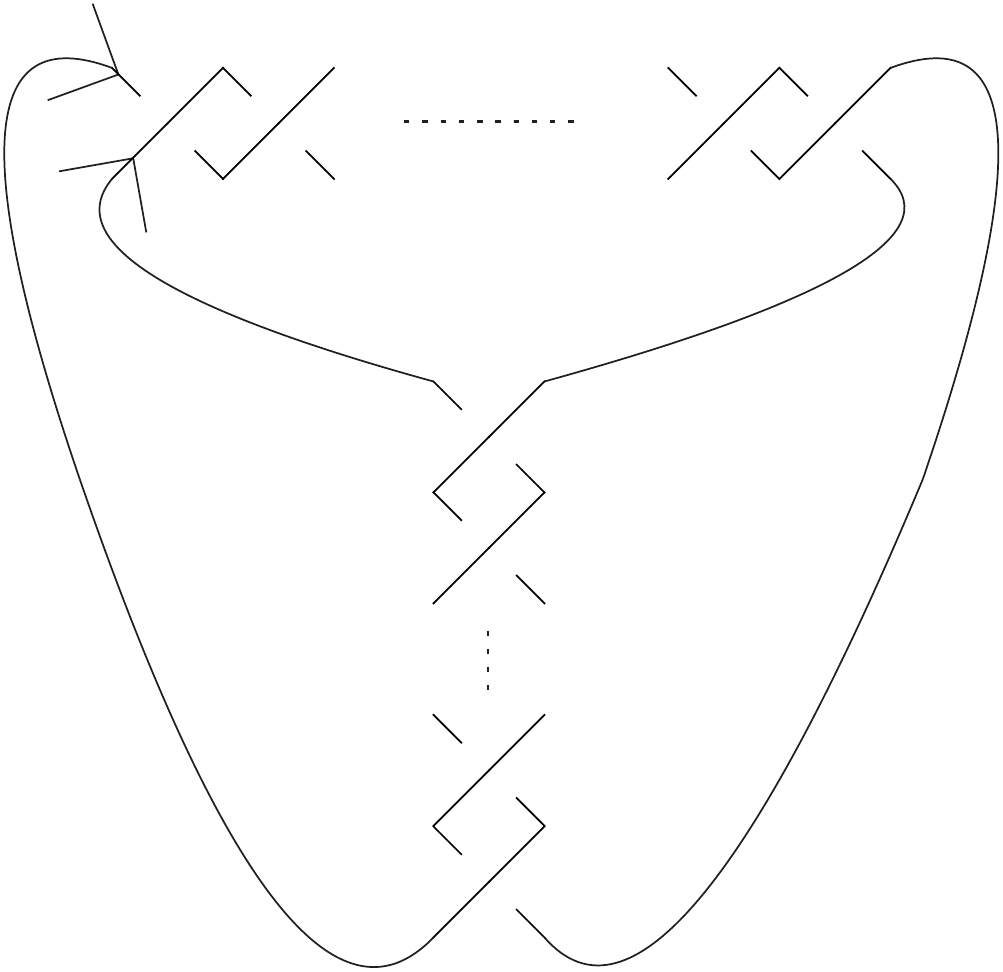}
                \caption{}
                \label{75ori}
        \end{subfigure}%
        ~ %add desired spacing between images, e. g. ~, \quad, \qquad, \hfill etc.
          %(or a blank line to force the subfigure onto a new line)
%        \begin{subfigure}[b]{0.3\textwidth}
%                \includegraphics[width=\textwidth]{33orientation1}
%                \caption{A tiger}
%                \label{fig:tiger}
%        \end{subfigure}
        ~ %add desired spacing between images, e. g. ~, \quad, \qquad, \hfill etc.
          %(or a blank line to force the subfigure onto a new line)
        \begin{subfigure}[b]{0.2\textwidth}
                \includegraphics[width=\textwidth]{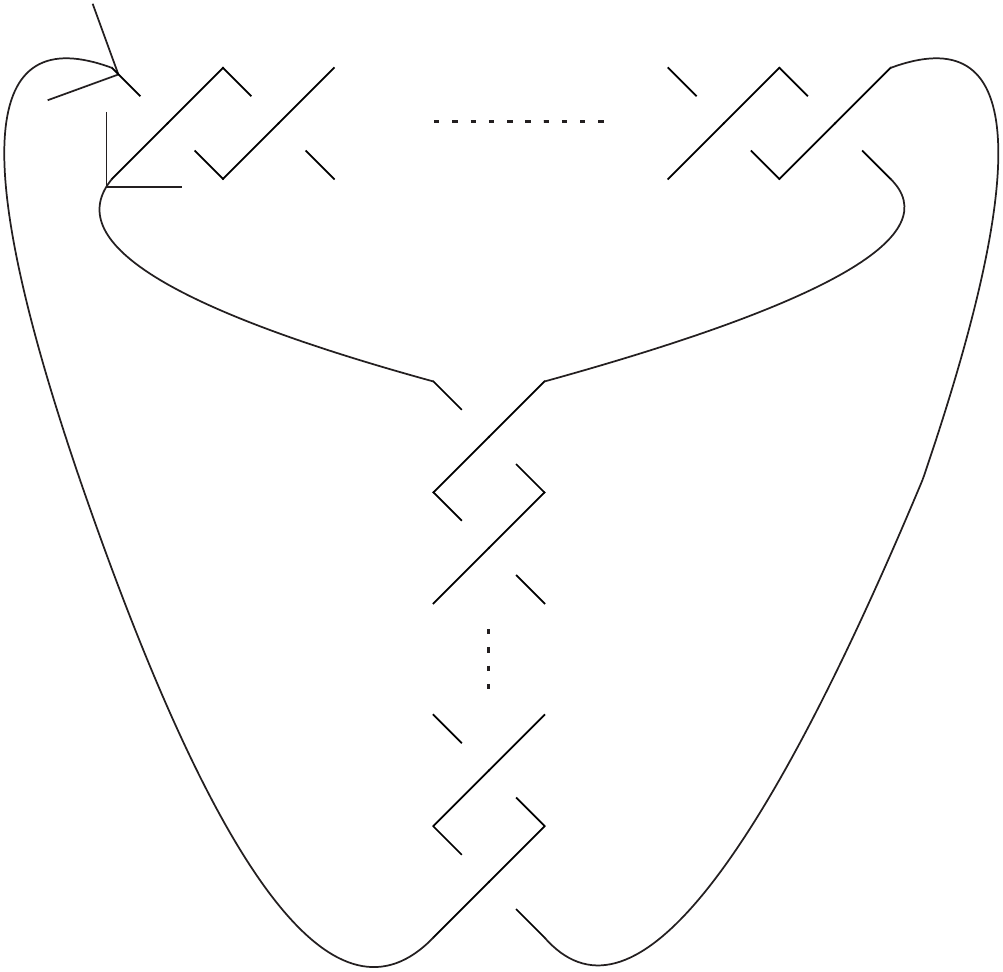}
                \caption{}
                \label{75ori2}
        \end{subfigure}
%        \caption{Pictures of animals}\label{fig:animals}
%\includegraphics[width=5cm,height=1cm]{parallel.pdf}
\vspace*{8pt}
\caption{$2$-bridge link $C(m,\ n)$ with different orientations}\label{mnlink}
\end{center}
\end{figure}

\noindent Note that $C(m,\ p,\ n)$ will be a knot (not link) if and only if either only one of $m$ or $n$ is even or all $m,\ p$ and $n$ are odd.

\begin{theorem}
Consider a $2$-bridge knot with Conway's notation $C(m,\ p,\ n)$
\begin{enumerate}
\item when only one of either $m$ or $n$ is even. Specifically, if $m$ is even
\begin{center}
\emph{Arf}$(C(m,\ p,\ n))=$
$\begin{cases}
0& \text{\emph{if} $\sum\limits_{i}A(R_i) \equiv 0$ $($\emph{{mod} 4}$ )$;}\\
1& \text{\emph{if} $\sum\limits_{i}A(R_i) \equiv 2$ $($\emph{{mod} 4}$ )$,}
\end{cases}$
\end{center}
where if
\begin{enumerate}
\item $p$ is even
\begin{center}
$\sum\limits_{i}A(R_i)=$
$\begin{cases}
2\lfloor\frac{m+n+2}{4}\rfloor& \text{\emph{if} $m \equiv 0$ $($\emph{{mod} 4}$ )$;}\\
p+2\lfloor\frac{m+n+2}{4}\rfloor& \text{\emph{if} $m \equiv 2$ $($\emph{{mod} 4}$ )$.}
\end{cases}$
\end{center}
\item $p$ is odd
\begin{center}
$\sum\limits_{i}A(R_i)=$
$\begin{cases}
\lfloor\frac{m}{2}\rfloor& \text{\emph{if} $m \equiv 0$ $($\emph{{mod} 4}$ )$;}\\
p+\lfloor\frac{m}{2}\rfloor& \text{\emph{if} $m \equiv 2$ $($\emph{{mod} 4}$ )$.}
\end{cases}$
\end{center}
\end{enumerate}

\item Otherwise, when all $m,\ p$ and $n$ are odd and if
\begin{enumerate}
\item $p+n\equiv \ 0\ (mod\ 4)$, then
\begin{center}
\emph{Arf}$(C(m,\ p,\ n))=$
$\begin{cases}
0& \text{\emph{if} $2\lfloor\frac{p+2}{4}\rfloor \equiv 0$ $($\emph{{mod} 4}$ )$;}\\
1& \text{\emph{if} $2\lfloor\frac{p+2}{4}\rfloor \equiv 2$ $($\emph{{mod} 4}$ )$.}
\end{cases}$
\end{center}
\item $p+n\equiv \ 2\ (mod\ 4)$, then
\begin{center}
\emph{Arf}$(C(m,\ p,\ n))=$
$\begin{cases}
0& \text{\emph{if} $m+1+2\lfloor\frac{p}{4}\rfloor \equiv 0$ $($\emph{{mod} 4}$ )$;}\\
1& \text{\emph{if} $m+1+2\lfloor\frac{p}{4}\rfloor \equiv 2$ $($\emph{{mod} 4}$ )$.}
\end{cases}$
\end{center}
\end{enumerate}
\end{enumerate}
\end{theorem}

\begin{proof}
Note that when $m,\ p$ are even and $n$ is odd then $a(c) = -1 = w(c)$ for each crossing $c$ in $C(m,\ p,\ n)$. When $m$ is even and  $n,\ p$ are odd and if crossing $c$ is in $t_m$ or $t_p'$ then $a(c) = -1 = w(c)$, otherwise if crossing $c$ is in $t_n$ then $a(c) = -1 = -w(c)$. Also when all $m,\ p$ and $n$ are odd, $a(c) = -1 = -w(c)$ for all crossings $c$ in $C(m,\ p,\ n)$. Theorem \ref{5.thm28} and Theorem \ref{5.thm31} provide the positions of region crossing changes to transform $C(m,\ p,\ n)$ to a trivial knot. Further result is based on Theorem \ref{5.thm.3.2} and simple calculation of $\sum\limits_{i}A(R_i)$. \hfill $\square$
\end{proof}

\end{document}